\pdfoutput=1
\RequirePackage{ifpdf}
\ifpdf % We are running pdfTeX in pdf mode
\documentclass[pdftex]{sigma}
\else
\documentclass{sigma}
\fi

\numberwithin{equation}{section}

\newtheorem{Theorem}{Theorem}[section]
\newtheorem{Corollary}[Theorem]{Corollary}
\newtheorem{Lemma}[Theorem]{Lemma}
\newtheorem{Proposition}[Theorem]{Proposition}
 { \theoremstyle{definition}
\newtheorem{Definition}[Theorem]{Definition}
\newtheorem{notation}[Theorem]{Notation}
\newtheorem{Example}[Theorem]{Example}
\newtheorem{Remark}[Theorem]{Remark} }

\DeclareMathOperator{\rig}{rig}
\DeclareMathOperator{\Aut}{Aut}
\DeclareMathOperator{\dR}{dR}
\DeclareMathOperator{\MW}{MW}
\DeclareMathOperator{\diag}{diag}
\DeclareMathOperator{\prim}{prim}
\DeclareMathOperator{\Frob}{Frob}

\begin{document}
\allowdisplaybreaks

\newcommand{\arXivNumber}{1706.01626}

\renewcommand{\thefootnote}{}

\renewcommand{\PaperNumber}{087}

\FirstPageHeading

\ShortArticleName{Zeta Functions of Monomial Deformations of Delsarte Hypersurfaces}

\ArticleName{Zeta Functions of Monomial Deformations\\ of Delsarte Hypersurfaces\footnote{This paper is a~contribution to the Special Issue on Modular Forms and String Theory in honor of Noriko Yui. The full collection is available at \href{http://www.emis.de/journals/SIGMA/modular-forms.html}{http://www.emis.de/journals/SIGMA/modular-forms.html}}}

\Author{Remke KLOOSTERMAN}

\AuthorNameForHeading{R.~Kloosterman}

\Address{Universit\`a degli Studi di Padova, Dipartimento di Matematica,\\ Via Trieste 63, 35121 Padova, Italy}
\Email{\href{mailto:klooster@math.unipd.it}{klooster@math.unipd.it}}
\URLaddress{\url{http://www.math.unipd.it/~klooster/}}

\ArticleDates{Received June 09, 2017, in f\/inal form November 01, 2017; Published online November 07, 2017}

\Abstract{Let $X_\lambda$ and $X_\lambda'$ be monomial deformations of two Delsarte hypersurfaces in weighted projective spaces. In this paper we give a~suf\/f\/icient condition so that their zeta functions have a common factor. This generalises results by Doran, Kelly, Salerno, Sperber, Voight and Whitcher [arXiv:1612.09249], where they showed this for a particular monomial deformation of a~Calabi--Yau invertible polynomial. It turns out that our factor can be of higher degree than the factor found in [arXiv:1612.09249].}


\Classification{14G10; 11G25; 14C22; 14J28; 14J70; 14Q10}

\renewcommand{\thefootnote}{\arabic{footnote}}
\setcounter{footnote}{0}

\section{Introduction}
Fix a f\/inite f\/ield $\mathbf{F}_{q}$ and a positive integer $n$. In this paper we study a particular class of deformations of Delsarte hypersurfaces in $\mathbf{P}^n_{\mathbf{F}_q}$. There has been an extensive study of the behaviour of the zeta function in families of varieties. First results were obtained by Dwork (e.g.,~\cite{DworkPadicCycle}) and Katz~\cite{Katz}. In the latter paper the author studies a pencil of hypersurface in $\mathbf{P}^n$ and describe a~dif\/ferential equation, whose solution is the Frobenius matrix on the middle cohomology for a~general member of this pencil.

More recently, the behaviour of the zeta function acquired renewed interest because of two interesting (and very dif\/ferent) applications. Candelas, de la Ossa and Rodriguez--Villegas \cite{Cand2} studied the behaviour of the zeta family in a particular family of quintic threefolds in~$\mathbf{P}^4$, with a particular interest in phenomena, analogous to phenomena occurring in characteristic zero related with mirror symmetry and let to many subsequent papers by various authors. Another application of Katz' dif\/ferential equation can be found in algorithms to determine the zeta function of a hypersurface ef\/f\/iciently (see \cite{LauDefT,PanTui}).

The main aim of this paper is to generalize and to comment on a recent result of Doran, Kelly, Salerno, Sperber, Voight and Whitcher \cite{DKSSVW} on the zeta function of certain pencils of Calabi--Yau hypersurfaces. For a more extensive discussion on the history of this particular result we refer to the introduction of \cite{DKSSVW}.

To describe the main results from \cite{DKSSVW}, f\/ix a matrix $A:=(a_{i,j})_{0\leq i,j\leq n}$ with nonnegative integral coef\/f\/icients and nonzero determinant. Then with $A$ we can associate the polynomial
\begin{gather*} F_A:=\sum_{i=0}^n\prod_{j=0}^n x_i^{a_{i,j}}.\end{gather*}
Assume that the entries of $A^{-1} (1,\dots,1)^{\rm T}$ are all positive, say $\frac{1}{e}(w_0,\dots,w_n)$ with $e,w_i\in \mathbf{Z}_{>0}$. Then $F_A$ def\/ines a hypersurface of degree $e$ in $\mathbf{P}_{\mathbf{F}_q}(w_0,\dots,w_n)$. Assume that we choose $A$ such that $\gcd(q,e)=1$. (Equivalently, we may assume that $\gcd(\det(A),q)=1$.)

If the hypersurface is geometrically irreducible then we call it a \emph{Delsarte hypersurface}. A~subvariety $X\subset \mathbf{P}(w_0,\dots,w_n)$ is called quasismooth if the af\/f\/ine quasicone of $X$ is smooth away from the vertex. If $F_A$ def\/ines a quasismooth hypersurface then $F_A$ is called an \emph{invertible polynomial}. If $F_A=0$ def\/ines a Calabi--Yau manifold, i.e., $e=n+1$, then we can consider the one-parameter family $X_{A,\psi}$ given by the vanishing of
\begin{gather*} F_A-(n+1)\psi \prod_{i=0}^n x_i.\end{gather*}
The factor $-(n+1)$ is included for historic reasons. In the sequel we will work with the parameter $\lambda=-(n+1)\psi$ for simplicity.

In a recent preprint Doran, Kelly, Salerno, Sperber, Voight and Whitcher \cite{DKSSVW} showed the following result (using Dwork cohomology and some results on the Picard--Fuchs equation):
\begin{Theorem}[\cite{DKSSVW}]\label{thmDKSSVW} Let $A$ and $A'$ be $(n+1)\times (n+1)$-matrices with nonnegative entries such that $F_A$ and $F_{A'}$ are invertible Calabi--Yau polynomials of degree $n+1$. Assume that $(1,\dots,1)^{\rm T}$ is an eigenvector of both $A$ and $A'$ and that \begin{gather*}\gcd(q,(n+1)\det(A)\det(A'))=1.\end{gather*}
Moreover, assume that $(1,\dots,1)A^{-1}$ and $(1,\dots,1){A'}^{-1}$ are proportional. Then for any $\psi\in \mathbf{F}_q$ such that $X_{A,\psi}$ and $X_{A',\psi}$ are smooth and nondegenerate we have that the polynomials
\begin{gather*} \left(Z(X_{A,\psi},T)\prod_{i=0}^{n-1}\big(1-q^iT\big) \right)^{(-1)^{n}}\qquad \mbox{and} \qquad
 \left(Z(X_{A',\psi},T)\prod_{i=0}^{n-1}\big(1-q^iT\big) \right)^{(-1)^{n}}\end{gather*}
have a common factor of degree at least the order of the Picard--Fuchs equation of~$X_{A,\psi}$.
\end{Theorem}
For a precise def\/inition of nondegenerate we refer to the paper \cite{DKSSVW}. The condition $(1,1,\dots,1)^{\rm T}$ is an eigenvector of~$A$ implies that $X_{A,\psi}\subset \mathbf{P}^n$. The condition $(1,\dots,1)A^{-1}$ is proportional to $(1,\dots,1){A'}^{-1}$ is the same as the condition \emph{dual weights being equal} from the paper~\cite{DKSSVW}, whenever the latter condition is def\/ined.

In this paper we prove a generalisation of this result. We aim to allow more matrices $A$, more vectors $\mathbf{a}$, to drop the Calabi--Yau assumption, to have a simpler nondegenerate assumption and to f\/ind a common factor of higher degree. Moreover, as a by-prodcut of our approach we obtain additional information on the degree of the factor found in \cite{DKSSVW}.

To be more precise, we start again with an invertible $(n+1)\times (n+1)$-matrix $A$ such that~$X_{A,0}$ is quasismooth, but we drop the Calabi--Yau condition. Let $d$ be an integer such that $B:=dA^{-1}$ has integral entries. Let $\mathbf{w}=(w_0,\dots,w_n):=B(1,\dots,1)^{\rm T}$. If all the $w_i$ are positive then $F_A$ def\/ines a hypersurface in the weighted projective space $\mathbf{P}(\mathbf{w})$.

Fix now a vector $\mathbf{a}:=(a_0,\dots,a_n)$ such that $a_i\in \mathbf{Z}_{>0}$, the entries of $\mathbf{b}:=\mathbf{a} B$ are nonnegative and $\sum\limits_{i=0}^n a_iw_i=d$. Then $ F_{A,\psi}:=F_A-(n+1)\psi \prod\limits_{i=0}^n x_i^{a_i}$ def\/ines a family of hypersurfaces $X_{\psi}$ in~$\mathbf{P}(\mathbf{w})$ each birational to a quotient of $Y_\psi\subset \mathbf{P}^n$ given by
\begin{gather*} \sum_{i=0}^n y_i^d-(n+1)\psi \prod_{i=0}^n y_i^{b_i}.\end{gather*}
This is a one-dimensional monomial deformation of a Fermat hypersurface. It is easy to determine for which values of~$\psi$ the hypersurface is smooth \cite[Lemma~3.7]{zetafam}. The idea to study Delsarte hypersurface by using their Fermat cover dates back to Shioda~\cite{ShiExpPic} and has then been used by many authors for to discuss solve problems concerning Delsarte hypersurfaces by con\-si\-de\-ring a similar problem on Fermat surfaces. Recent applications of this idea, in contexts similar to our setup, can be found in \cite{BiniSH,BGK, KellyMirror}.

Take now a further $(n+1)\times(n+1)$ matrix $A'$ and a vector $\mathbf{a}'$ yielding a second family~$X'_{\psi}$ in a possibly dif\/ferent weighted projective space.

It is straightforward to show that if $\mathbf{a} A^{-1}$ and $\mathbf{a}' (A')^{-1}$ are proportional then the fami\-lies~$X_\psi$,~$X'_\psi$ have a common cover of the type $Y_\psi$, i.e., there exist subgroup schemes~$G$ and~$G'$ of the scheme of automorphisms $\Aut(Y_\psi)$ such that $G$ and $G'$ are def\/ined over $\mathbf{F}_q$, $Y_\psi/G$ is birational to $X_\psi$ and $Y_\psi/G'$ is birational to~$X'_{\psi}$. The automorphisms in~$G$ and $G'$ are so-called torus or diagonal automorphisms, i.e., each automorphism multiplies a coordinate with a root of unity. In particular, $G$~and~$G'$ are f\/inite abelian groups. We will use this observation to show that:

\begin{Theorem}\label{thmMain} Let $A$ and $A'$ be $(n+1)\times (n+1)$-matrices with nonnegative entries, such that the entries of $ (w_0,\dots,w_n)^{\rm T}:=A^{-1} (1,\dots,1)^{\rm T}$ and of $(w'_0,\dots,w'_n)^{\rm T}:={A'}^{-1}(1,\dots,1)^{\rm T}$ are all positive and $\gcd(q,\det(A)\det(A'))=1$. Fix two vectors $\mathbf{a}:=(a_0,\dots,a_n)$ and $\mathbf{a}':=(a'_0,\dots,a'_n)$ consisting of nonnegative integers such that the equalities $\sum\limits_{i=0}^n a_iw_i=1$ and $\sum\limits_{i=0}^n a'_iw'_i=1$ hold and such that $\mathbf{a} A$ and $\mathbf{a}' A'$ are proportional. Let $X_\psi$, $X'_\psi$, $Y_\psi$, $G$ and $G'$ as above. Denote with $G.G'$ the subgroup of $\Aut(Y_\psi)$ generated by $G$ and~$G'$.
\begin{enumerate}\itemsep=0pt
\item[$1.$] If $Y_\psi$ is smooth then the characteristic polynomial of Frobenius acting on $H^{n-1}(Y_{\psi})^{G.G'}$ divides both the characteristic polynomial of Frobenius acting on $H^{n-1}(X_\psi)$ and the cha\-rac\-teristic polynomial of Frobenius on~$H^{n-1}(X'_\psi)$.

\item[$2.$] If, moreover, $(1,\dots,1)$ is an eigenvector of both $A$ and $A'$ and both $X_\psi$ and $X'_\psi$ are smooth then we have that the polynomials
\begin{gather*} \left(Z(X_{\psi},T)\prod_{i=0}^{n-1}\big(1-q^iT\big) \right)^{(-1)^{n}}\qquad
\text{and} \qquad \left(Z(X'_{\psi},T)\prod_{i=0}^{n-1}\big(1-q^iT\big) \right)^{(-1)^{n}}\end{gather*}
have a common factor of positive degree.
\end{enumerate}
\end{Theorem}

In the second section we will prove this result under slightly weaker, but more technical hypothesis, see Theorem~\ref{thmMainO} and Corollary~\ref{corMain}. Moreover, in Proposition~\ref{prpFactorDiv} we will show that the factor constructed in the proof of Theorem~\ref{thmDKSSVW} divides the characteristic polynomial of Frobenius acting on $H^{n-1}(Y_\psi)^{G.G'}$. We will give examples where our factor has higher degree.

Note that the quotient map $Y_\psi\dashrightarrow X_{\psi}$ is a rational map. If it were a morphism then it is straightforward to show that the characteristic polynomial of $H^{n-1}(Y_{\psi})^{G.G'}$ divides the characteristic polynomial of Frobenius on $H^{n-1}(X_\psi)$. Hence, large part of the proof is dedicated to show that passing to the open where the rational map is a morphism does not kill any part of $H^{n-1}(Y_{\psi})^{G.G'}$.

In the course of the proof of Theorem~\ref{thmMain} we show that we can decompose $H^{n-1}(X_\psi)$ as a~direct sum of two Frobenius stable subspaces, namely
\begin{gather*} H^{n-1}(X_\psi)=H^{n-1}(Y_\psi)^G \oplus C.\end{gather*}
Similarly, we show that can decompose $H^{n-1}(Y_{\psi})^{G}= H^{n-1}(Y_{\psi})^{G_{\max}} \oplus W_\psi$, where $W_\psi$ is Frobenius stable, and $G_{\max}$ is the maximal group of torus automorphisms acting on the family $Y_\psi$.

The appearance of $C$ is related with the fact that the quotient map is only a rational map rather than a morphism. For most choices of $(a_0,\dots,a_n)$ we have that $C$ is independent of $\psi$ and in that case we can express $C$ in terms of the cohomology of cones over Fermat hypersurfaces. Hence the Frobenius action on~$C$ is easy to determine. To calculate the Frobenius action on the complementary subspace $H^{n-1}(Y_{\psi})^G$ we can use the methods from~\cite{zetafam} to express the zeta function in terms of generalised $p$-adic hypergeometric functions.

This brings us to another observation from \cite{DKSSVW}: In \cite[Section~5]{DKSSVW} the authors consider f\/ive families of quartic $K3$ surfaces which have a single common factor of the zeta function of degree~3. They show that every other zero of the characteristic polynomial of Frobenius on $H^2$ is of the form $q$ times a root of unity. Assuming the Tate conjecture for $K3$ surfaces (which is proven for most $K3$ surfaces anyway) we deduce that the (geometric) Picard number is at least~19.

This result is a special case of the following phenomena: if for a lift to characteristic zero $h^{n-1,0}\big(H^n(Y_\psi)^{G_{\max}}\big) =h^{n-1,0}\big(H^{n-1}(X_{\psi})\big)$ holds then it turns out that both $W_\psi$ and $C$ are Tate twists of Hodge structures of lower weight. In the $K3$ case, $W_\psi$ and $C$ are Hodge structures of pure $(1,1)$-type. By the Lefschetz' theorem on $(1,1)$-classes, they are generated by classes of divisors. In particular, for each of the f\/ive families the lifts to characteristic zero have Picard number at least~19, and since they form a~one-dimensional family the generic Picard group has rank~19.

In the second half of the paper we discuss how one can f\/ind a basis for a subgroup of f\/inite index of the generic Picard group for the f\/ive families from \cite{DKSSVW} and for f\/ive further monomial deformations of Delsarte quartic surfaces. For all ten families we determine $H^2(Y_\psi)^G$, $H^2(Y_\psi)^{G_{\max}}$ and $C$ as vector spaces with Frobenius action Moreover, we f\/ind curves generating $C$ in each of the ten cases. For two families we have that $W_{\psi}$ is zero-dimensional. For six of the remaining eight families we manage to f\/ind curves, whose classes in cohomology generate $W_{\psi}$.

In the next section we prove our generalisation of the result from \cite{DKSSVW}. In the third section we discuss the quartic surface case. In Appendix~\ref{appBit} we give explicit equations for bitangents to certain particular quartic plane curves. These equations can be used to f\/ind explicit curves, generating $W_{\psi}$.

\section{Delsarte hypersurface}

Fix an integer $n\geq 2$ and f\/ix a f\/inite f\/ield $\mathbf{F}_q$.
\begin{Definition}An invertible matrix $A:=(a_{i,j})_{0\leq i,j\leq n}$, such that all entries are nonnegative integers is called a \emph{coefficient matrix} if all entries of $A^{-1}(1,\dots,1)^{\rm T}$ are positive and each column of $A$ contains a zero.

In that case let $d$ be an integer such that $B:=dA^{-1}$ has integral coef\/f\/icients. We call $B$ the \emph{map matrix}. We call $B(1,\dots,1)^{\rm T}$ the \emph{weight vector}, which we denote by $\mathbf{w}:=(w_0,\dots,w_n)$.

A vector $\mathbf{a}:=(a_0,\dots,a_n)$ consisting of nonnegative integers such that $\sum\limits_{i=0}^n w_ia_i=d$ holds and such that all entries of $\mathbf{a} A^{-1} $ are nonnegative is called a \emph{deformation vector}.
\end{Definition}

\begin{Definition} Fix a pair $(A,\mathbf{a})$ consisting of coef\/f\/icient matrix and a deformation vector $\mathbf{a}$. Assume that $\gcd(q,d)=1$. Then we call $(A,\mathbf{a})$ \emph{Delsarte deformation data of length~$n$}.
\end{Definition}

Let $(A,\mathbf{a})$ be Delsarte deformation data of length $n$. Let
\begin{gather*} X_\lambda:=Z\left( \sum_{i=0}^n \prod_{j=0}^n x_j^{a_{i,j}} + \lambda \prod_{i=0}^n x_i^{a_i}\right)\end{gather*}
be the corresponding one-parameter family of hypersurfaces of weighted degree $d$ in the weighted projective space $\mathbf{P}(w_0,\dots,w_n)$.

Denote with $(b_0,\dots,b_n)$ the entries of $\mathbf{a} B $. Let $Y_{\lambda}$ be
\begin{gather*} Z\left(\sum_{i=0}^n y_i^d +\lambda \prod_{i=0}^n y_i^{b_i}\right) \subset \mathbf{P}^n.\end{gather*}

\begin{Remark} Our def\/inition of $\mathbf{w}$ may lead to choices of the $w_i$ such that the gcd of $(w_0,\dots,w_n)$ is larger than one. The choice of the $w_i$ is such that the weighted degree of the polynomial def\/ining $X_{\lambda}$ equals the degree of $Y_\lambda$.
\end{Remark}

We have a $(\mathbf{Z}/d\mathbf{Z})^{n}$-action on $\mathbf{P}^n$ induced by
\begin{gather*} (g_1,\dots,g_n)(x_0:x_1:\dots:x_n):=\big(x_0:\zeta^{g_1}x_1:\zeta^{g_2}x_2:\dots:\zeta^{g_n}x_n\big),\end{gather*}
with $\zeta$ a f\/ixed primitive $d$-th root of unity. The subgroup $G$ def\/ined by $\sum\limits_{i=1}^n g_i b_i\equiv 0 \bmod d$ acts on $Y_{\lambda}$.

The rational map $\mathbf{P}^n\dashrightarrow \mathbf{P}(\mathbf{w})$ given by
\begin{gather*} (y_0,y_1,\dots,y_n) \mapsto \left(\prod_{i=0}^n y_i^{b_{0,i}},\prod_{i=0}^n y_i^{b_{1,i}},\dots,\prod_{i=0}^n y_i^{b_{n,i}}\right)\end{gather*}
induces a rational map $Y_{\lambda}\dashrightarrow X_{\lambda}$. This rational map is Galois (i.e., the corresponding extension of function f\/ields is Galois) and the Galois group is a subgroup of~$G$.

In particular, if all the $b_{i,j}$ are nonnegative then this rational map is a morphism. (This map was used by Shioda \cite{ShiExpPic} to give an algorithm to calculate the Picard number of a Delsarte surface in~$\mathbf{P}^3$.)

\begin{Lemma}
The hypersurface $X_0$ is irreducible.
\end{Lemma}
\begin{proof}
 Each column of $A$ contains a zero by the def\/inition of coef\/f\/icient matrix. Hence $x_k$ does not divide \begin{gather*}\sum_{i=0}^n \prod_{j=0}^n x_j^{a_{i,j}}\end{gather*}
for any $k$. Hence for every irreducible component of $X_0$ the points such that all coordinates are nonzero are dense, and these latter points are in the image of $Y_0$. This implies that every irreducible component of $X_0$ is the closure of an irreducible component of the image of $Y_0$. Since $n>1$ it follows that $Y_0$ is irreducible and hence $X_0$ is irreducible.
\end{proof}

\begin{Definition}
We call $X_0$ the \emph{Delsarte hypersurface} associated with $A$ and $X_\lambda$ the \emph{one-dimen\-sio\-nal monomial deformation} associated with $(A,\mathbf{a})$. If, moreover, $X_0$ is quasismooth then we call $X_0$ \emph{invertible hypersurface}.
\end{Definition}

\begin{Example} Consider
\begin{gather*} x_0^4+x_1^4+x_2^3x_3+x_3^3x_2+\lambda x_0x_1x_2x_3.\end{gather*}
Then
\begin{gather*} A=\left(\begin{matrix} 4&0&0&0\\0&4&0&0\\0&0&3&1\\0&0&1&3\end{matrix} \right) \qquad \text{and} \qquad \mathbf{a}=(1,1,1,1).\end{gather*}
We have that
\begin{gather*} B=\left(\begin{matrix} 2&0&0&0\\0&2&0&0\\0&0&3&-1\\0&0&-1&3\end{matrix} \right) \qquad \text{and} \qquad \mathbf{w}=(2,2,2,2). \end{gather*}
In particular, we have that this family is birational a quotient of
\begin{gather*} x_0^8+x_1^8+x_2^8+x_3^8+\lambda (x_0x_1x_2x_3)^2.\end{gather*}
The group $G$ is generated by the automorphisms
\begin{gather*} (x_0,x_1,x_2,x_3)\mapsto(x_0,-x_1,x_2,x_3)\qquad \text{and} \qquad
(x_0,x_1,x_2,x_3)\mapsto\big(x_0,x_1,\zeta^3 x_2,\zeta x_3\big),\end{gather*} with $\zeta$ a primitive $8$-th root of unity.
\end{Example}

\begin{Definition}\label{defGenPos} A hypersurface $X=V(f)\subset \mathbf{P}^n$ is \emph{in general position} if
\begin{gather*} V\left(x_0\frac{\partial}{\partial{x_0}} f, \dots,x_n\frac{\partial}{\partial{x_n}} f\right)\end{gather*}
is empty. Equivalently, $X$ is smooth and for any subset $\{i_1,\dots,i_c\} \subset \{0,1,\dots,n\}$ we have that
\begin{gather*} X \cap V(x_{i_1})\cap \dots \cap V(x_{i_c})\end{gather*}
is also smooth.
\end{Definition}

\begin{Lemma}\label{lemGenPos} If $Y_\lambda$ is smooth then $Y_{\lambda}$ is in general position.
\end{Lemma}

\begin{proof} Suppose we intersect $Y_{\lambda}$ with $x_{i_1}=\dots=x_{i_c}=0$. If some $b_{i_j}$ is nonzero then the inter\-section is a Fermat hypersurface in $\mathbf{P}^{n-c}$ and is smooth. If all $b_{i_j}$ are zero then we can do the following: After a change of coordinates we may assume that $\{i_1,\dots,i_{c}\}=\{0,1,\dots,c-1\}$. We now have that $Y_{\lambda}$ is the zero set of
\begin{gather*} \sum_{i=0}^{c-1} x_i^d+h(x_c,\dots,x_n)\end{gather*}
for some $h\in \mathbf{F}_q[x_c,\dots,x_n]$. From $\gcd(q,d)=1$ it follows that the singular points of the intersection $V(x_0,\dots,x_{c-1},h)$ are in one-to-one correspondence with the singular points of $Y_\lambda$. Hence $V(x_0,\dots,x_{c-1},h)$ is smooth.
\end{proof}

Recall that we started with a hypersurface $X_\lambda \subset \mathbf{P}(\mathbf{w})$ and constructed a hypersurface $Y_{\lambda}\subset \mathbf{P}^n$, such that $X_\lambda$ is birational to a quotient of $Y_\lambda$. Denote with $U_{\lambda}:=\mathbf{P}(\mathbf{w})\setminus X_\lambda$ and $V_{\lambda}:=\mathbf{P}^n\setminus Y_\lambda$ be the respective complements.

Denote now with $(\mathbf{P}(\mathbf{w}))^*$, $U_{\lambda}^*$, $V_{\lambda}^*$, $X_{\lambda}^*$, $Y_\lambda^*$, etc. the original variety minus the intersection with $Z(x_0\dots x_n)$ or $Z(y_0 \dots y_n)$, the union of the coordinate hyperplanes. We have that the quotient map $\mathbf{P}^n \dashrightarrow \mathbf{P}(\mathbf{w})$ def\/ines surjective morphisms $(\mathbf{P}^n)^*\to \mathbf{P}(\mathbf{w})^*$, $Y_\lambda^*\to X_{\lambda}^*$, $V_\lambda^*\to U_\lambda^*$.

There is a second quotient map $\mathbf{P}^n \to \mathbf{P}(\mathbf{w})$ given by $(z_0:\dots:z_n) \to (z_0^{w_0}:\dots:z_n^{w_n})$. This map is a morphism and is a ramif\/ied Galois covering. Denote with $H$ the corresponding Galois group. Let $\tilde{X}_{\lambda}$ be the pull back of $X_\lambda$ and let $\tilde{U}_{\lambda}$ be the pull back of $U_{\lambda}$.

Fix now a lift $\mu\in \mathbf{Q}_q$ of $\lambda$. Then we can def\/ine $F_{\mu}$, $\tilde{U}_{\mu}$, $V_{\mu}$, $\tilde{ X}_{\mu}$, $Y_{\mu}$ similarly as above. If $y_0,\dots,y_n$ are projective coordinates on $\mathbf{P}^n$ then let $\Omega$ be
\begin{gather*} \left(\prod_{i=0}^n y_i\right) \left(\sum_{i=0}^n (-1)^{i} \frac{dy_0}{y_0} \wedge \dots \wedge \widehat{\frac{dy_i}{y_i} }\wedge\dots \wedge\frac{dy_n}{y_n} \right).\end{gather*}
We recall now some standard notation used to study the cohomology of a hypersurface complement in $\mathbf{P}^n$.

\begin{notation}Let $\mathbf{m}=(m_0,\dots,m_n)$ be $(n+1)$-tuple of positive integers, such that $\sum\limits_{i=0}^n m_i =td$ for some positive integer~$t$. Then
\begin{gather*} \tilde{\omega}_{\mathbf{m}}:= \frac{\prod\limits_{i=0}^n x_i^{m_i-1}}{(F_{A,\mu}(x_0^{w_0},\dots,x_n^{w_n}))^t} \Omega\end{gather*}
is an $n$-form on the complement $\tilde{U}_{\mu}$ of $\tilde{X}_{\mu}$. If we allow the $m_i$ and $t$ to be arbitrary integers such that the equality $\sum\limits_{i=0}^n m_i=td$ holds then $\tilde{\omega}_{\mathbf{m}}$ is a form on $\tilde{U}^*_{\mu}$.

Let $\mathbf{m}=(m_0,\dots,m_n)$ be $(n+1)$-tuple of positive integers, such that $\sum\limits_{i=0}^n m_i =td$ holds for some positive integer $t$. Let $D$ be the diagonal matrix $d I_{n+1}$. Then
\begin{gather*} \omega_{\mathbf{m}}:= \frac{\prod\limits_{i=0}^n y_i^{m_i-1}}{F_{D,\mu}^t} \Omega\end{gather*}
is an $n$-form on the complement $V_{\mu}$ of $Y_{\mu}$. If we allow the $m_i$ and $t$ to be arbitrary integers such that the equality $\sum\limits_{i=0}^n m_i=td$ then $\omega_{\mathbf{m}}$ is a form on~$V^*_{\mu}$.
\end{notation}

The following result seems to be known to the experts, but we include it for the reader's convenience:

\begin{Lemma}\label{lemBasis} There exists a finite set $S\subset \mathbf{Q}_q$ such that $0\not \in S$ and for all $\mu\in \mathbf{Q}_q\setminus S$ we have that
\begin{gather*}\mathcal{B}:=\left \{\omega_{\mathbf{m}} \colon 0<m_i<d \mbox{ for } i=0,\dots, n \ \text{and} \ \sum_{i=0}^n m_i\equiv 0 \bmod d\right \}\end{gather*}
is a basis for $H^n_{\dR}(V_\mu,\mathbf{Q}_q)$.

Similarly, there exists a finite set $S^*$ such that $0\not \in S^*$ and for all $\mu\in\mathbf{Q}_q\setminus S^*$ we have that
\begin{gather*}\mathcal{B}^*:=\left \{\omega_{\mathbf{m}} \colon 0\leq m_i<d \mbox{ for } i=0,\dots, n \ \text{and} \ \sum_{i=0}^n m_i\equiv 0 \bmod d\right\}\end{gather*}
is a basis for $H^n_{\dR}(V_\mu^*,\mathbf{Q}_q)$.
\end{Lemma}
\begin{proof} The forms $\omega_{\mathbf{m}}$, such that $m_i \geq 1$ for $i=0,\dots,n$ generate the de Rham cohomology group $H^n_{\dR}(V_\mu)$.
By dif\/ferentiating certain particular $(n-1)$-forms on $V_\mu$ we have that the following relation in $H^n_{\dR}(V_\mu)$
\begin{gather}\label{eqnRed} \frac{G_{y_i}}{F^t} \Omega = \frac{t G F_{y_i}}{F^{t+1}} \Omega \end{gather}
for any form $G\in \mathbf{Q}_q[y_0,\dots,y_n]_{td-n}$. (This is the so-called Grif\/f\/iths--Dwork method to reduce forms in cohomology.)

For $\mu=0$ we have that $F_{y_i}=dy_i^{d-1}$. Using (\ref{eqnRed}) we f\/ind the relation \begin{gather} \label{eqnRedb} \frac{y_0^{m_0} G(y_1,\dots,y_n)}{F_0^{t+1}}\Omega = \frac{(m_0-d+1)y_0^{m_0-d} G(y_1,\dots,y_n)}{t F_0^t} \Omega\end{gather}
and similar relations for the other $y_i$. In this way we can reduce forms such that all exponents are at least $0$ and at most $d-1$. However, if an exponent equals $d-1$ then this relation yields that the class is zero in cohomology. In particular, the $\omega_{\mathbf{m}}$ with $0<m_i<d$ for $i=0,\dots, n$ and $\sum\limits_{i=0}^n m_i\equiv 0 \bmod d$ generate $H^n_{\dR}(V_0)$. Grif\/f\/iths \cite{GriRat} showed that the relations of type~(\ref{eqnRed}) generate all relations and hence $\mathcal{B}$ is a basis for $H^n_{\dR}(V_0)$. If $X_\mu$ is smooth then the dimension of $H^n_{\dR}(V_\mu)$ is independent of $\mu$ and it is then straightforward to check that there are at most f\/initely many choices of $\mu$ for which $X_\mu$ is smooth and $\mathcal{B}$ is not a basis for $H^n_{\dR}(V_\mu)$.

We now prove the statement on $H^n_{\dR}(V_\mu^*)$. Note that if $X_\mu$ is smooth then by Lemma~\ref{lemGenPos} it is in general position. Therefore the dimension of $H^n_{\dR}(V_\mu^*)$ is independent of $\mu$. Hence it suf\/f\/ices to show that $\mathcal{B}^*$ is a basis for $H^n_{\dR}(V_0^*)$. Again we have relations of type~(\ref{eqnRed}), but now we may take $G\in \mathbf{Q}_q\big[y_0,y_0^{-1},\dots,y_n,y_n^{-1}\big]_{td-n}$. If the exponent of a~variable is at most $-2$ then we can use the relations of the shape~(\ref{eqnRedb}) to increase the exponent of this variable. However, if the exponent equals $-1$ we cannot do this, because then we would have to divide by zero in~(\ref{eqnRedb}). In this way we obtain that $\mathcal{B}^*$ generates $H^n_{\dR}(V^*_0)$. Moreover, as in the above case there are no further relations and~$\mathcal{B}^*$ is a basis.
\end{proof}

\begin{Remark} The $\mu$ for which $\mathcal{B}$ is not a basis can be determined by the methods of \cite[Section~3]{PanTui}.
\end{Remark}

\begin{Remark}Denote with $H^n_{\MW}(V_{\lambda})$ the $n$-th Monsky--Washnitzer cohomology of $V_\lambda$. The Monsky--Washnitzer cohomology is essentially the cohomology of the tensor product of de~Rham complex of a lift of $V_{\lambda}$ to characteristic zero with a weakly complete f\/initely generated algebra~$A^{\dagger}$. The Frobenius action on the cohomology is induced by a lift of Frobenius to $A^{\dagger}$. For more details see \cite[Theorem~2.4.5]{PutMW}. In that paper it is shown that two dif\/ferent lifts of $V_{\lambda}$ yield isomorphic complexes and two choices of lifts of Frobenius yield homotopic maps on the complexes. In particular, $H^n_{\MW}(V_\lambda)$ is independent of the choices made.

Let $\mu\in \mathbf{Q}_q$ be a lift of $\lambda$. One choice of a lift of $V_\lambda$ to characteristic zero is $V_{\mu}$, and the construction of the Monsky--Washnitzer cohomology yields a natural map $H^n_{\dR}(V_\mu)\to H^n_{\MW}(V_\lambda)$ of $\mathbf{Q}_q$-vector spaces. If $Y_\lambda$ is smooth then the is an isomorphism by \cite{BalChi}. Since $V_\lambda \subset \mathbf{P}^n$ is af\/f\/ine and smooth we have an isomorphism $H^n_{\MW}(V_\lambda)\cong H^n_{\rig}(V_\lambda)$ by~\cite{BerSMF}, where the latter group is rigid cohomology. Since there are inf\/initely many lifts~$\mu$ of $\lambda$ we can always choose a lift~$\mu$ such that~$\mathcal{B}$ is a basis for $H^n_{\dR}(V_\mu)$ and thereby yielding a basis for $H^n_{\MW}(V_{\lambda})$.

If $Y_\lambda$ is smooth then using Lemma~\ref{lemGenPos} we f\/ind that $V_\lambda^*$ is the complement of a normal crossing divisor. In particular, we can apply \cite{BalChi} and f\/ind a natural isomorphism $H^n_{\MW}(V_\lambda^*)\cong H^n_{\dR}(V_\mu^*)$. As above, we have an isomorphism $H^n_{\MW}(V_\lambda^*)\cong H^n_{\rig}(V_\lambda^*)$ and we identif\/ied a basis for $H^n_{\rig}(V_\lambda^*)$.
\end{Remark}

\begin{Remark} The action of $G$ lifts to characteristic zero. The forms $\omega_{\mathbf{m}}$ are eigenvectors for~$g^*$ each element $g\in G$. Hence the $G$-invariant ones span $H^n_{\MW}(V_\lambda)^G$.

If $Y_{\lambda}$ is singular then by the def\/inition of Monsky--Washnitzer cohomology we have that $H^n_{\MW}(V_\lambda)$ is generated by expressions
\begin{gather*} \sum_{\mathbf{m}=(m_0,\dots,m_n),m_i\geq 1} a_{\mathbf{m}} \omega_{\mathbf{m}}\end{gather*}
such that there exists $c_1,c_2 \in \mathbf{Q}$ with $c_1>0$ and $v(a_{\mathbf{m}})\geq c_1\big(\sum\limits_{i=0}^n m_i\big)+c_2$. The space $H^n_{\MW}(V_\lambda^*)$ is generated by expression
\begin{gather*} \sum_{\mathbf{m}=(m_0,\dots,m_n),m_i\geq -N} a_{\mathbf{m}} \omega_{\mathbf{m}}\end{gather*}
such that there exists $c_1,c_2 \in \mathbf{Q}$ with $c_1>0$ and $v(a_{\mathbf{m}})\geq c_1\big(\sum\limits_{i=0}^n m_i\big)+c_2$. The $G$-invariant subspaces are generated by similar sums, but in which only the $G$-invariant~$\omega_{\mathbf{m}}$ occur.
\end{Remark}

\begin{Remark} If one wants to study the Frobenius matrix by using the dif\/ferential equations, like in~\cite{Katz} or in~\cite{DKSSVW} then one needs to be more careful in lifting $V_\lambda$ to characteristic zero. In~\cite{Katz} one has to take~$\mu$ to be the Teichm\"uller lift of~$\lambda$. The reason for this, is that a priori Frobenius maps $H^i\big(U_\mu^q\big)$ to $H^i(U_\mu)$. To have an operator on $H^n(U_\mu)$ we need that $\mu^q=\mu$. If one works directly with Frobenius on Monsky--Washnitzer chomology then this constraint on $\mu$ does not exist.
\end{Remark}

From now on we use $H^i$ and $H^i_c$ to indicate rigid cohomology respectively rigid cohomology with compact support.

By \cite[Proposition 2.1]{Zetafamrec} we have canonical isomorphisms
\begin{gather*} H^i_c(U_\lambda^*)\cong H^i_c(V_\lambda^*)^G \qquad \text{and} \qquad H^i_c(U_\lambda^*)\cong H^i_c\big(\tilde{U}_\lambda^*\big)^H.\end{gather*}
We want to compare the cohomology of $H^n_c(V_\lambda)^G$ with the cohomology of~$H^n_c(U_{\lambda})$. However, $U_\lambda$ may be singular, hence we work with $H^n_c\big(\tilde{U}_\lambda\big)^ H$ instead. Using Poin\-car\'e duality it suf\/f\/ices to compare $H^n(V_\lambda)^ G$ with $H^n\big(\tilde{U}_\lambda\big)^H$ instead. Since both varieties are smooth and af\/f\/ine we can identify their rigid cohomology groups with their Monsky--Washnitzer cohomology groups. We will do this in order to prove:

\begin{Proposition}\label{prpFactor} Suppose that $H^n(V_\lambda)\to H^n(V_{\lambda}^*)$ is injective. Then $H^n(V_{\lambda})^G$ is a quotient of $H^n\big(\tilde{U}_{\lambda}\big)^H$. In particular, the characteristic polynomial of Frobenius acting on $H^n_c(V_{\lambda})^G$ is in~$\mathbf{Q}[T]$ and divides the characteristic polynomial of Frobenius acting on~$H^n_c(U_{\lambda})$.
\end{Proposition}

\begin{proof}Since $H^n(V_\lambda)\to H^n(V_{\lambda}^*)$ is injective we have by \cite{BerPoi} that the Poincar\'e dual of this map is surjective, and therefore that $H^n_c(V_\lambda)^G$ is a quotient of~$H^n_c(V_\lambda^*)^G$. This implies that $H^n_c(V_\lambda)^G$ is also a quotient of $H^n_c\big(\tilde{U}^*_\lambda\big)^H$. Hence it suf\/f\/ices to show that the kernel of natural map $H^n_c\big(\tilde{U}^*_\lambda\big)\to H^n_c\big(\tilde{U}_\lambda\big)$ is mapped to zero in $H^n_c(V_\lambda)^G$. Using Poincar\'e duality we can consider $H^n(V_\lambda)^G$ as a subspace of $H^n\big(\tilde{U}_{\lambda}^*\big)^H$. It suf\/f\/ices to show that $H^n(V_\lambda)^G$ is in the image of~$H^n(\tilde{U}_\lambda)$.

A form $\omega_{\mathbf{k}}$ is in $H^n(V_\lambda)^G$ if and only if there is a monomial type $\mathbf{m}_0$ such that $\mathbf{k}=\mathbf{m}_0 B$. We identif\/ied $H^n(V_\lambda)^G$ with a subspace of $H^n\big(\tilde{U}_\lambda^*\big)^H$. The class of $\omega_{\mathbf{k}}$ is identif\/ied with $\tilde{\omega}_{\mathbf{m}}$ where $\mathbf{m}=\mathbf{m}_0(\diag(w_0,\dots,w_n))$.

The entries of $\mathbf{m}$ are integers, which may be nonpositive. If all entries of $\mathbf{m}$ are positive then~$\omega_{\mathbf{m}}$ is in the image of $H^n\big(\tilde{U}_{\lambda}\big)^ H$. Recall that $B=d A^{-1}$ and therefore $\mathbf{m}_0=\mathbf{k} \frac{1}{d} A $. Since~$\mathbf{k}$ has positive entries, $A$ has positive entries and no zero column it follows that also the entries of~$\mathbf{m}_0$ are positive and therefore all entries of~$\mathbf{m}$ are also positive. This yields the f\/irst statement.

To prove the second statement. By \cite[Lemma~4.3]{Zetafamrec} it follows that $H^n_c(V_\lambda)^G$ is Frobenius invariant and the characteristic polynomial is in $\mathbf{Q}[T]$. Using Poincar\'e duality we f\/ind that~$H^n_c(V_\lambda)^G$ is a subspace of~$H^n_c(\tilde{U}_\lambda)^H$. As explained above, the latter space is isomorphic with~$H^n_c(U_\lambda)$.
\end{proof}

\begin{Definition}
Fix Delsarte deformation data $(A_1,\mathbf{a}_1), \dots, (A_t,\mathbf{a}_t)$ of length~$n$. We say that they have a \emph{common cover}
if for every $i$, $j$ we have that $\mathbf{a}^{\rm T}_i A_i^{-1}$ and $\mathbf{a}^{\rm T}_j A_j^{-1}$ are proportional.
\end{Definition}
\begin{Example} Take the following f\/ive matrices
\begin{gather*}
\left(\begin{matrix} 4&0&0&0\\0&4&0&0\\0&0&4&0\\0&0&0&4\end{matrix} \right)\!,\quad
\left(\begin{matrix} 4&0&0&0\\0&4&0&0\\0&0&3&1\\0&0&1&3\end{matrix} \right)\!,\quad
\left(\begin{matrix} 3&1&0&0\\1&3&0&0\\0&0&3&1\\0&0&1&3\end{matrix} \right)\!,\quad
\left(\begin{matrix} 4&0&0&0\\0&3&1&0\\0&0&3&1\\0&1&0&3\end{matrix} \right)\!,\quad
\left(\begin{matrix} 3&1&0&0\\0&3&1&0\\0&0&3&1\\1&0&0&3\end{matrix} \right)\!.
\end{gather*}
In each case we take $(1,1,1,1)^{\rm T}$ as the deformation vector then the $(A_i,\mathbf{a}_i)$ have a common cover.
\end{Example}

Suppose now that $(A_1,\mathbf{a}_1),\dots,(A_t,\mathbf{a}_t)$ have a common cover. Let $d$ be the smallest positive integer such that $dA_i^{-1}$ has integral coef\/f\/icients for all~$i$. The sum of the entries of $\mathbf{b}_i=\mathbf{a}_i \big(dA_i^{-1}\big)$ equals~$d$. By assumption we have that for each $i$ and $j$ the vectors $\mathbf{b}_i$ and $\mathbf{b}_j$ are proportional, hence these vectors coincide and we denote this common vector with~$\mathbf{b}$.

Denote with $b_j$ the entries of $\mathbf{b}$. Denote with $X_{i,\lambda}$ the family associated with $(A_i,\mathbf{a}_i)$. Then $X_{i,\lambda}$ is birational to a quotient of
\begin{gather*}Y_\lambda\colon \ \sum_{i=0}^n y_i^{d} +\lambda \prod_{i=0}^n y_i^{b_i}.\end{gather*}
At the beginning of this section we gave an explicit description of this map. From that description it follows that $Y_\lambda \to X_{i,\lambda}$ is def\/ined whenever all the~$y_i$ are nonzero.

We can now apply Proposition~\ref{prpFactor} to the above setup and we f\/ind directly that:

\begin{Theorem} \label{thmMainO} Let $(A_1,\mathbf{a}_1), \dots, (A_t,\mathbf{a}_t)$ be Delsarte deformation data of length $n$ with a common cover. Denote with~$X_{i,\lambda}$ be the corresponding families of Delsarte hypersurfaces and with~$Y_\lambda$ the common cover. Let $G_i$ be the Galois group of the function field extension corresponding to the rational map $Y_\lambda\dashrightarrow X_{i,\lambda}$. Then the automorphisms in~$G_i$ extend to automorphisms of~$Y_{\lambda}$. Identify~$G_i$ with the corresponding subgroup of $\Aut(Y_{\lambda})$. Let $G=G_1.G_2.\dots.G_t\subset \Aut(Y_\lambda)$.

Suppose that $H^n(V_\lambda)\to H^n(V_\lambda^*)$ is injective. Then for each $i=1,\dots,t$ we have that $H^n_c(V_{\lambda})^G$ is a quotient of $H^n_c(U_{i,\lambda})$. In particular, the characteristic polynomial of Frobenius on~$H^{n-1}(Y_\lambda)^G$ is in $\mathbf{Q}[T]$ and is a common factor of the characteristic polynomials of Frobenius acting on~$H^{n-1}(X_{i,\lambda})$.
\end{Theorem}

\begin{Remark} Recall that in order to be Delsarte deformation data we need that $\gcd(q,(n+1)\det(A_i))=1$ for all $i$.
\end{Remark}

\begin{Remark} If $Y_\lambda$ is smooth then it is in general position by Lemma~\ref{lemGenPos}.

The map $H^ {n-1}(Y_{\lambda})\to H^{n-1}(Y_{\lambda}^*)$ is injective if $n-1$ is even, and has a kernel if $n-1$ is odd, and this kernel is generated by the hyperplane class, see \cite[Theorem~1.19]{Katz}. The residue map identif\/ies $H^{n}(V_{\lambda})$ with the primitive part of the cohomology of $H^{n-1}(Y_{\lambda})$. In particular, the composition $H^n(V_{\lambda})\to H^{n-1}(Y_{\lambda}^*)$ is injective independent of the parity of $n$. From the diagram on \cite[p.~79]{Katz} it follows that the latter map factors through $H^n(V_\lambda^*)$. In particular, $H^n(V_\lambda)\to H^n(V_\lambda^*)$ is injective. Hence we can apply the above proposition if $Y_\lambda$ is smooth. The values of $\lambda$ for which $Y_\lambda$ is singular can be determined from the formula \cite[Lemma~3.7]{zetafam}.
\end{Remark}

To conclude that there is a common factor of the zeta function is more complicated in general. The zeta function is a quotient of products of characteristic polynomials of Frobenius and there may be some cancellation in this quotient. However, if we make the extra assumptions that each~$X_{i,\lambda}$ is a hypersurface in~$\mathbf{P}^n$ (i.e., for each $i$ we have that $\mathbf{w}=(k,\dots,k)$ for some $k\in \mathbf{Z}_{>0}$) and we consider only values of $\lambda$ for which $X_\lambda$ is smooth then we have that
 \begin{gather*} \left(Z(X_{i,\lambda},T)\prod_{j=0}^{n-1} \big(1-q^j T\big)\right)^{(-1)^n} = \det\big(I- T \mathrm{Frob}^* \colon H^{n}_c(U_{i,\lambda})\big).\end{gather*}
From the smoothness of $X_{i,\lambda}$ it follows that the eigenvalues of Frobenius on~$H^{n}_c(U_{i,\lambda})$ have absolute value~$q^{{n-1}/2}$, hence there is no cancellation in this formula and we obtain:

\begin{Corollary}\label{corMain} Let $(A_1,\mathbf{a}_1), \dots, (A_t,\mathbf{a}_t)$ be Delsarte deformation data of length $n$ with a common cover. Denote with $X_{i,\lambda}$ be the corresponding families of Delsarte hypersurfaces and with~$Y_\lambda$ the common cover. Let~$G_i$ be the Galois group of the function field extension corresponding to $Y_\lambda\to X_{i,\lambda}$. Let $G=G_1.G_2.\dots.G_t\subset \Aut(Y_\lambda)$.

Suppose that for each $i$ we have $\mathbf{P}(\mathbf{w})=\mathbf{P}^n$. Moreover, suppose that $Y_{\lambda}$ and each $X_{i,\lambda}$ is smooth. Then the characteristic polynomial of Frobenius on $H^{n-1}(Y_\lambda)_{\prim}^G$ is in $\mathbf{Q}[T]$ and divides the polynomial
\begin{gather*} \left( {Z(X_{i,\lambda},T)}{\prod_{j=0}^{n-1} \big(1-q^jT\big)} \right)^{(-1)^n}.\end{gather*}
\end{Corollary}

\begin{Remark} A complex hypersurface with quotient singularities is a $\mathbf{Q}$-homology manifold and satisf\/ies Poincar\'e duality. The existence of Poincar\'e duality is suf\/f\/icient to obtain both the vanishing statement $H^i_c (\mathbf{P}(\mathbf{w})\setminus X_{i,\lambda})=0$ for $i\neq n,2n$ as well as for the purity statement on~$H^n_c(\mathbf{P}(\mathbf{w})\setminus X_{i,\lambda})$.

Hence if Poincar\'e duality would hold for the rigid cohomology of varieties with (tame) quotient singularities over f\/inite f\/ields then we could extend the above corollary to the case where $X_{i,\lambda}$ is a quasi-smooth hypersurface.\end{Remark}

We would like to compare our factor with the factor found in \cite{DKSSVW}. The groups $G$ and $G'$ consists of torus automorphisms of $Y_\lambda$. Let $G_{\max}$ be the group of torus automorphism of~$Y_{\lambda}$. Then $G_{\max}$ is an abelian group. A torus automorphism $g\in G_{\max}$ sends $Y^*_\lambda$ to itself, and descents to an automorphism of $X^*_\lambda\cong Y^*_{\lambda}/G$. Hence the quotient group~$G_{\max}/G$ acts on~$X^*_\lambda$.

Since the quotient map is given by $n+1$ monomials we have that a torus automorphism descends to a torus automorphism of~$X^*_\lambda$ and~$U^*_{\lambda}$. Any torus automorphism can be extended to~$\mathbf{P}(\mathbf{w})$, leaving $X_\lambda$ invariant. Hence we have an action of $G_{\max}/G$ on $H^n(U_\lambda)$ and on $H^n_c(U_{\lambda})$. It is straightforward to check that $G_{\max}/G\cong {\rm SL}(F_A)$, where ${\rm SL}(F_A)$ is the group introduced in~\cite{DKSSVW}, and that both groups act the same.

The factor from \cite{DKSSVW} is constructed as follows: The authors identify a subspace of the Dwork cohomology group $H^n_{\rm Dwork}(U_{\lambda})^{{\rm SL}(F_A)}$, whose dimension equals the order of the Picard--Fuchs equation of $X_{\lambda}$ and which is invariant under Frobenius. They show that the characteristic polynomial~$R'_{\lambda}$ of Frobenius on this subspace is in $K[T]$ for some number f\/ield $K$, which can be taken Galois over~$\mathbf{Q}$ and then take~$R_{\lambda}$ the be the least common multiple of the Galois conjugates of~$R'_{\lambda}$.

To compare this polynomial with the factor constructed above, we will start by reconside\-ring~$R'_{\lambda}$, i.e., we will show that it is just the characteristic polynomial of Frobenius acting on
\begin{gather*} H^n_c(V_\lambda)^{G_{\max}}. \end{gather*}
Then \cite[Lemma~4.3]{Zetafamrec} implies that $R'_{\lambda}\in \mathbf{Q}[T]$ and that $R_\lambda=R'_\lambda$.

We start by calculating the dimension of $H^n_c(V_\lambda)^{G_{\max}}$.
\begin{Lemma} \label{lemPF} Suppose that $X_{\lambda}$ is Calabi--Yau, i.e., $\sum w_i=d$ and suppose that~$Y_\lambda$ is smooth. Then the dimension of $H^n_c(V_\lambda)^{G_{\max}}$ equals the order of the Picard--Fuchs equation for~$X_\lambda$.
\end{Lemma}

\begin{proof}Since $V_\lambda$ is smooth we have by Poincar\'e duality \cite{BerPoi} that \begin{gather*}\dim H^n_c(V_\lambda)^{G_{\max}}=\dim H^n(V_\lambda)^{G_{\max}}.\end{gather*}
We now calculate the latter dimension. The group $G_{\max}$ consists of the $(g_1,\dots,g_n)$ in $ (\mathbf{Z}/d\mathbf{Z})^n$ such that \begin{gather*}\sum_{i=1}^n g_ib_i\equiv 0 \bmod d.\end{gather*} From \cite[Lemma~4.2]{Zetafamrec} it follows that $G_{\max}$ f\/ixes the dif\/ferential form $\omega_{\mathbf{k}}$ if and only if $\mathbf{k}\equiv t\mathbf{b}$ $\bmod~d$ for some $t\in \mathbf{Z}/d\mathbf{Z}$. Hence $H^n(V_{\lambda})^{G_{\max}}$ is spanned by $\omega_{t\mathbf{b}}$ where $t\in \{0,1,\dots,d-1\}$ such~$t\mathbf{b} \bmod d$ has no zero entry.

The number of $t \in \mathbf{Z}/d\mathbf{Z}$ such that $t\mathbf{b} \bmod d$ has a zero entry equals the number of $t \in \{0,\dots,d-1\}$ for which there exists an~$i$ and an integer~$k$ such that $t b_i=k d$, or, equivalently,
\begin{gather*} \frac{t}{d}=\frac{k}{b_i}.\end{gather*}
Since $0\leq t<d$ we may assume that $ 0\leq k <b_i$. Using the notation from \cite[Section~2]{DKSSVW} we have that the elements on the left hand side are in the set they call $\alpha$ and the elements on the right hand side are in the set $\beta$. In particular, the number of $t$ such that $t\mathbf{b}$ has no zero entry equals $d-\# \alpha \cap \beta$. G\"ahrs \cite[Theorem~2.8]{GaehrsPhD} showed that this number equals the order of the Picard--Fuchs equation.
\end{proof}

\begin{Proposition}\label{prpFactorDiv} Suppose $w_0=\dots=w_n=1$, $d=n+1$ and $a_i=1$ for $i=0,\dots,n$. Then the factor $R'_{\lambda}$ found in~{\rm \cite{DKSSVW}} is the characteristic polynomial of Frobenius acting on~$H^n(V_{\lambda})^{G_{\max}}$. In particular, $R'_{\lambda}\in \mathbf{Q}[T]$ and $R_\lambda=R'_\lambda$.
\end{Proposition}
\begin{proof} Since $\mathbf{P}(\mathbf{w})=\mathbf{P}^n$ we have $U_{\lambda}=\tilde{U}_\lambda$. Hence we can discuss dif\/ferential forms on the complement of~$X_{\lambda}$.

The factor $R'_\lambda(T)$ obtained in \cite{DKSSVW} using the $p$-adic Picard--Fuchs equation in Dwork coho\-mo\-logy. The main result from~\cite{Katz} yields a dif\/ferential equation satisf\/ied by the Frobenius operator on $H^n_{\MW}(U_\lambda,\mathbf{Q}_q)$ and that this dif\/ferential equation can also be found using Dwork cohomology. In particular, $R'_\lambda(T)$ is the characteristic polynomial of Frobenius acting on the subspace $P$ containing $\omega_{\mathbf{a}}$ and invariant under the Picard--Fuchs operator. This subspace~$P$ is contained in the span of~$\{\omega_{s\mathbf{a}} \colon s=1,2,\dots\}$.

Pick a form $\tilde{\omega}_{t\mathbf{a}}$ restrict this form to $U_{\lambda}^*$ and then pull it back to a form on $V_{\lambda}^*$. Then this pull back is $\omega_{t\mathbf{b}}$. This form is def\/ined on all of $V_{\lambda}$. Hence the pullback of $P$ to $H^n(V_{\lambda})$ is well-def\/ined and is contained in $H^n(V_{\lambda})^{G_{\max}}$. Hence $P$ is a subspace of $H^n(V_{\lambda})^{G_{\max}}$. Since both spaces have the same dimension by Lemma~\ref{lemPF} they coincide, i.e., $P\cong H^n(V_{\lambda})^{G_{\max}}$ as vector spaces with Frobenius action.

Now $R_{\lambda}$ is the characteristic polynomial of $q^n \Frob^{-1}$ acting on $P$. Using Poincar\'e duality this equals the characteristic polynomial of Frobenius acting on $H^n_c(V_\lambda)^{G_{\max}}$. This yields the f\/irst claim. The obtained polynomial is in $\mathbf{Q}[T]$ by \cite[Lemma~4.3]{Zetafamrec} and hence $R'_{\lambda}(T)=R_\lambda(T)$.
\end{proof}

\section{Case of quartic surfaces}\label{secQua}
In this section we consider the case of invertible quartic polynomials. Up to permutation of the coordinates there are 10 invertible quartic polynomials in four variables. For each of these quartics we take $\mathbf{a}=(1,1,1,1)$ as the deformation vector.

In Fig.~\ref{tbl} we list the 10 families, which we denote here with $X^{(i)}_\lambda$. We provide the following information in the table. In the column~``$d$'' we list the minimal degree of a Fermat cover of the central f\/iber. When we discuss one of the examples we always assume that $\gcd(q,d)=1$. In the next column we list the deformation vector $\mathbf{b}:=(1,1,1,1)^{\rm T} B$. Hence the corresponding Fermat cover $Y_\lambda^{(i)}$ is def\/ined by
 \begin{gather*} x_0^d+x_1^d+x_2^d+x_3^d+\lambda x_0^{b_0}x_1^{b_1}x_2^{b_2}x_3^{b_3}.\end{gather*}
Let $G$ be the Galois group of the function f\/ield extension corresponding to the morphism \smash{$Y_{\lambda}^*\to X_{\lambda}^*$}. The next two columns deal with $H^3(V_\lambda)^G$, for $\lambda$ such that $Y_\lambda$ is smooth. In the column $PF$ we list the dimension of~$H^3(V_{\lambda})^{G_{\max}}$. We calculated this entry as follows: from the results from \cite[Section~4]{Zetafamrec} it follows that a basis for this vector space consists of those $\omega_{\mathbf{k}}$ such that all entries of $\mathbf{k}$ are between $1$ and $d-1$ and there is a $t\in \mathbf{Z}$ such that $\mathbf{k} \equiv t \mathbf{b} \bmod d$. It is straight forward to determine the number of these~$\mathbf{k}$. As discussed in the previous section, this number equals the order of the Picard--Fuchs equation of~$X_{\lambda}^{i}$.

The next column concerns the subspace $W^{(i)}_\lambda \subset H^3\big(V^{(i)}_{\lambda}\big)^G$. The $\omega_{\mathbf{k}}$ such that each of the entries of $\mathbf{k}$ is in $\{1,\dots,d-1\}$ and there exists a vector $\mathbf{m} \in \mathbf{Z}^4$ such that $\mathbf{k}\equiv \mathbf{m} A \bmod d$ form a basis for $H^3\big(V_{\lambda}^{(i)}\big)^G$. For each of the 10 examples we checked for each $\mathbf{k}$ if such a~$\mathbf{m}$ existed or not and used this to calculate $\dim W_{\lambda}^{(i)}=\dim H^3\big(V_{\lambda}^{(i)}\big)^G-\dim H^3\big(V^{(i)}_{\lambda}\big)^{G_{\max}}$.

For the families $1,2,3,6,7$ we listed all these $\mathbf{k}$ in Fig.~\ref{tblform} (they are enlisted in the corresponding column in Fig.~\ref{tblform}, in the f\/irst column there is a choice for a possible $\mathbf{m}$, the forms marked with $(PF)$ are in $H^3(V_{\lambda})^{G_{\max}}$). For $i=4,5,8,9,10$ we will describe $W_{\lambda}^{(i)}$ in the examples below.

Finally, the column $c$ then equals $21-\dim H^3(V_\lambda)^G$. As we argued in the introduction the subspace $C$ of dimension $c$ and $W_\lambda^{(i)}$ are generated by classes of curves on $X^{(i)}_\lambda$. For the families $i=1,2,3,6,7,9,10$ we give a recipe to f\/ind linear combinations of curves on~$X^{(i)}_\lambda$, which generate~$C$ and~$W_\lambda^{(i)}$. In fact, for all~$i$ we have that~$C$ is generated by curves, each of which is contained in one of the coordinate hyperplanes. These curves are easy to f\/ind for each~$i$. For $i=9,10$ we have that $W_{\lambda}^{(i)}=0$. For $i=1,2,3,6,7$ we can f\/ind various del Pezzo surfaces of degree $2$ together with morphisms of degree 2, such that linear combinations of pull backs of curves from these del Pezzo surfaces generated~$W_\lambda^{(i)}$. For $i=5$ we have a similar procedure using del Pezzo surfaces of degree $5$.
\begin{figure}[t]
\begin{gather*}
\begin{array}{|c|c|c|c|c|c|c|}
\hline
i& F_0 & d& (1,1,1,1)^{\rm T}\vphantom{\big|^2} B & PF &\dim W_\lambda & c \\
\hline
1&x_0^4+x_1^4+x_2^4+x_3^4 & 4&(1,1,1,1) & 3 & 18 & 0 \\
 2&x_0^4+x_1^4+x_2x_3\big(x_2^2+x_3^2\big) & 8&(2,2,2,2) &3 & 12 & 6 \\
 3& x_0x_1\big(x_0^2+x_1^2\big)+x_2x_3\big(x_2^2+x_3^2\big) &8&(2,2,2,2) &3 & 10 & 8 \\
4& x_0^4+x_1x_2^3+x_2x_3^3+x_3x_1^3, & 28 &(7,7,7,7) &3 & 18 & 0 \\
5&x_0x_1^3+x_1 x_2^3+x_2x_3^3+x_3x_0^3& 80 &(20,20,20,20) &3& 16 & 2 \\
 6&x_0^4+x_1^4+x_2^3x_3+x_3^4 & 12&(3,3,4,2) &6 & 12 & 3 \\
 7& x_0x_1\big(x_0^2+x_1^2\big)+x_2^3x_3+x_3^4 & 24&(6,6,8,4) & 6 & 8 & 7 \\
 8&x_0^3x_1+x_1^4+x_2^3x_3+x_3^4 & 12&(4,2,4,2) &4 &12 & 5 \\
 9&x_0^4+x_1^3x_2+x_2^3x_3+x_3^4 & 36&(9,12,8,7) &18 & 0 & 3\\
 10&x_0^3x_1+x_1^3x_2+x_2^3x_3+x_3^4 & 108&(36,24,28,20) &18 & 0 & 3 \\
 \hline
\end{array}
\end{gather*}\vspace{-5mm}

\caption{The 10 families $X^{(i)}_\lambda$.}\label{tbl}
\end{figure}

The f\/irst f\/ive families have a single common cover, also the sixth and seventh family have a~common cover. The common factor of the f\/irst f\/ive examples has degree~3. However, the f\/irst three examples have a~common factor of degree 5 and the f\/irst and the second example have a~common factor of degree~7.

\begin{figure}[t]
\begin{gather*}
\begin{array}{|c|c|c|c|c|c|}
\hline
\mathbf{m} & 3.1& 3.2 & 3.3 & 3.6 & 3.7 \\
 &&&&A=10, \, B=11 &\\
\hline
1111& 1111(PF)&2222 (PF) &2222(PF) & 3342 (PF) &6,6,8,4 (PF)\\
1124 & -&-& - &338A (PF) &6,6,16,20 (PF) \\
1133& 1133&2266 & 2266 &- &-\\
1214 & -&- & -&364B & 3,15,8,22\\
1223& 1223&2437 & 1537 & 3687 &3,15,16,14\\
1232& 1232&2473 & 1573 & - & -\\
1313& 1313&- & - & 3948 & -\\
1322& 1322&2644 & - & 3984 & -\\
1331& 1331&- & - & - & -\\
2114 & -&- & - &634B &15,3,8,22\\
2123& 2123&4237 & 5137 & 6387 &15,3,16,14\\
2132& 2132&4273 & 5173 & - & -\\
2222 & 2222(PF)&4444 (PF) &4444(PF) & 6684 (PF) & 12,12,16,8 (PF)\\
2213& 2213&- & - & 6648 (PF) & 12,12,8,16 (PF)\\
2231& 2231&- & - & - & -\\
2312& 2312&4615 & 3715 & 6945 &9,21,8,10\\
2321& 2321&4651 & 3751 & 6981&9,21,16,2\\
3113& 3113&- & - & 9348& -\\
3122& 3122&6244 & - & 9384 & -\\
3131& 3131&- & - & - & -\\
3212& 3212&6415 & 7315 &9645 &21,9,8,10\\
3221& 3221& 6451 & 7351 & 9681 & 21,9,16,2\\
3311& 3311&6622 & 6622 & 9942(PF)& 18,18,8,4 (PF)\\
3324& -& - & -& 998A (PF) & 18,18,16,20 (PF)\\
3333& 3333(PF)& 6666 (PF) & 6666(PF) & - &-\\
\hline
\end{array}
\end{gather*}\vspace{-5mm}

\caption{Generators for $H^3\big(V^{(i)}_\lambda\big)^G$.} \label{tblform}
\end{figure}

The following proposition now shows that claim about $W_{\lambda}^{(i)}$ for $i=1,2,3,6,7$:

\begin{Proposition}\label{prpdelPezzo} Consider one of the families $X_{\lambda}^{(i)}$ with $i\in \{1,2,3,6,7\}$ from Fig.~{\rm \ref{tbl}}. Then there exist families of del Pezzo surfaces~$S^{(i,j)}_{\lambda}$ and degree $2$ morphisms $\varphi^{(i,j)}_{\lambda}\colon X_{\lambda}^{(i)} \to S^{(i,j)}_{\lambda}$ such that if $i\in \{1,6,7\}$ then for almost all $\lambda$ we have that
\begin{gather*} W_\lambda^{(i)} \subset \sum_j \varphi^{(i,j)*}_{\lambda} \big(H^2\big(S^{(i,j)}_{\lambda}\big)\big)\end{gather*}
and if $i\in \{2,3\}$ then for almost all $\lambda$ we have that
\begin{gather*} W_\lambda^{(i)} \cap \sum_j \varphi^{(i,j)*}_{\lambda} \big(H^2\big(S^{(i,j)}_{\lambda}\big) \big)\end{gather*}
has codimension $2$ in $W_{\lambda}^{(i)}$ and the forms $\tilde{\omega}_{1133}$, $\tilde{\omega}_{3311}$ generate a complementary subspace in~$W_\lambda^{(i)}$.

If $i=3$ or $i=7$ or $q\equiv 1 \bmod 4$ then we can take the $S^{(i,j)}_{\lambda}$ to be defined over~$\mathbf{F}_q$. If $q\equiv 3 \bmod 4$ and $i\in\{1,2,6\}$ then some of the $S^{(i,j)}_{\lambda}$ are only defined over~$\mathbf{F}_{q^2}$.
\end{Proposition}

\begin{proof} Note that $W_{\lambda}^{(i)}$ is spanned by $\tilde{\omega}_{\mathbf{m}}$ where $\mathbf{m}$ are precisely these entries from the f\/irst column of Fig.~\ref{tblform} such that in the column corresponding to~$i$ there the entry is dif\/ferent from~``$-$'' and is without the mark~``$(PF)$''. Note also that in the notation of the previous section we have $\tilde{U}^{(i)}_{\lambda}={U}^{(i)}_{\lambda}$. Hence we denote dif\/ferential forms on the complement of $X_\lambda^{(i)}$ with $\tilde{\omega}$ and forms on the complement of $Y_{\lambda}^{(i)}$ with $\omega$.

A def\/ining polynomial for $X^{(i)}_{0}$ can be found in Fig.~\ref{tbl}. Recall that for each family we took $(1,1,1,1)$ as the deformation factor. In particular, each of the f\/ive families under consideration is each invariant under the automorphisms $\sigma$ and $\tau$ def\/ined by \begin{gather*} \sigma(x_0,x_1,x_2,x_3):=(x_1,x_0,x_2,x_3),\qquad \tau(x_0,x_1,x_2,x_3):=(-x_1,-x_0,x_2,x_3).\end{gather*} A straightforward calculation shows that the quotients of $X_{\lambda}^{(i)}$ by $\sigma$ and by $\tau$ are both surfaces of degree 4 in $\mathbf{P}(1,1,1,2)$, and that for general $\lambda$ they are smooth (explicit equations for these surfaces can be found in the appendix). Hence the quotient surfaces are del Pezzo surfaces of degree $2$. Denote the corresponding surfaces with $S^{(i,1)}_{\lambda}$ and~$S^{(i,2)}_{\lambda}$

Let $\mathbf{m}=(a,b,c,d)$ with $a,b,c,d\in \{1,2,3,4\}$ be such that $a+b+c+d\equiv 0 \mod 4$. Then $\tilde{\omega}_\mathbf{m}+\sigma^*\tilde{\omega}_{\mathbf{m}}$ is invariant under $\sigma^*$ and therefore contained in $\pi_1^*\big(H^2\big(S^{(i,1)}_{\lambda}\big)\big)$. Similarly, $\tilde{\omega}_\mathbf{a}+\tau^*\tilde{\omega}_{\mathbf{a}}$ is contained in $\pi_2^*\big(H^2\big(S^{(i,2)}_{\lambda}\big)\big)$. Now $\tilde{\omega}_\mathbf{m}+\sigma^*\tilde{\omega}_{\mathbf{m}}=\tilde{\omega}_{abcd}-\tilde{\omega}_{bacd}$ and $\tilde{\omega}_\mathbf{m}+\tau^*\tilde{\omega}_{\mathbf{m}}=\omega_{abcd}+(-1)^{a+b+1}\omega_{bacd}$. Hence, if $a+b$ is odd then $\tilde{\omega}_{\mathbf{m}}\in \pi_1^*\big(H^2\big(S^{(i,1)}_{\lambda}\big)\big)+\pi_2^*\big(H^2\big(S^{(i,2)}_{\lambda}\big)\big)$. In the case $i=7$ we have that $W_{\lambda}^{(i)}$ is generated by forms $\tilde{\omega}_{\mathbf{m}}$ with $a+b$ odd and we f\/inished this case. In the case $i=3$ we have that~$W_{\lambda}^{(i)}$ is generated by forms with $a+b$ odd and the two forms $\tilde{\omega}_{1133}$ and~$\tilde{\omega}_{3311}$. Hence we f\/inishes also this case.

In the remaining cases $i=1,2,6$ we have a further automorphism \begin{gather*}\tau_1\colon \ (x_0,x_1,x_2,x_3)\mapsto (Ix_1,-Ix_0,x_2,x_3),\end{gather*} where $I^2=-1$. Denote $S^{3,i}_{\lambda}$ the quotient by $\tau_1$. If $q\equiv 1\bmod 4$ then $S^{3,i}_\lambda$ is def\/ined over $\mathbf{F}_q$, but if $q\equiv 3 \bmod 4$ then it is only def\/ined over $\mathbf{F}_{q^2}$.

Note that \begin{gather*}\tau_1^*(\tilde{\omega}_{abcd})=(-1)^{b+1}(I)^{a+b} \tilde{\omega}_{bacd}.\end{gather*} Hence if $b$ is odd and $a+b\equiv 0 \bmod 4$ then $\tilde{\omega}_{abcd}+\tilde{\omega}_{bacd}$ is f\/ixed under $\tau_1^*$ and, as above, we f\/ind that $\tilde{\omega}_{abcd}$ is in $ \pi^*_1(H^2(S^{i,1}_{\lambda}))+\pi^*_3(H^2(S^{i,3}_{\lambda}))$.

Using Fig.~\ref{tblform} we can conclude that we recovered any $\tilde{\omega}_{\mathbf{m}}$ such that the f\/irst two entries are distinct. This f\/inishes the proof for the case $i=6$. In the case $i=2$, we only miss the forms~$\tilde{\omega}_{1133}$ and~$\tilde{\omega}_{3311}$, hence we are also done in this case. In the case $i=1$ there is a $S_4$ symmetry we can use. We recover all~$\tilde{\omega}_{\mathbf{k}}$ with at least two distinct entries in~$\mathbf{k}$ and this f\/inishes also this case.
\end{proof}

\begin{Remark} In the cases $i=2,3$ we do not recover $\tilde{\omega}_{1133}$ and $\tilde{\omega}_{3311}$. However, the families $i=1,2,3$ have
\begin{gather*} x_0^{24}+x_1^{24}+x_2^{24}+x_3^{24}+\lambda\left(x_0x_1x_2x_3\right)^6\end{gather*}
as a common cover. For each of three families the form $\tilde{\omega}_{1133}$ is pulled back to the form $\omega_{6,6,18,18}$ on $V_\lambda$. Hence we can use $X^{(1)}_\lambda$ to express this form in terms of divisors pulled back form~$S^{1,3}_{\lambda}$.
\end{Remark}

In the following examples we discuss how to f\/ind generators for the subspace~$C$. For the examples $i=4,5,8$ we list a basis for $H^3\big(V_\lambda^{(i)}\big)^{G_{\max}}$ and also discuss strategies to f\/ind generators for~$W_\lambda$.

\begin{Example} For the case $i=1,2,3,6,7$ we note that Proposition~\ref{prpdelPezzo} yields a basis for $W^{(i)}_\lambda$ in terms of curves pulled back from del Pezzo surfaces $S^{(i,j)}$. In the appendix we will explain how to f\/ind these curves.

For each of these cases we can f\/ind generators for $C$ in each of these cases, but the approach depends on~$i$:
\begin{itemize}\itemsep=0pt
\item[$i=1$] $C=0$ in this case.
\item[$i=2$] The curves given by $x_3=0$, $x_0=-I^k x_1$ and the ones given by $x_4=0$, $x_0=-I^kx_1$ are in $C$, with $I^2=-1$. One easily checks that they generate~$C$. These curves can also be obtained by pulling back curves from the del Pezzo quotients: For example, consider the quotient by the automorphism $\sigma\colon (x_0,x_1,x_2,x_3)\mapsto (x_1,x_0,x_2,x_3)$. We f\/ind that $\omega_{abcd}-\omega_{bacd}$ is a $1$ eigenvector if $(a,b)\neq (b,a)$. However, there are only~5 such eigenvectors. The Picard group of the del Pezzo surface has rank 8. The additional three divisors are the hyperplane class and the two curves pulled back from the curves $x_3=x_1^2+x_2^2=0$ and $x_4=x_1^2+x_2^2=0$.
\item[$i=3$] We have that for $j<2$ and $k>1$ the line $x_j=x_k=0$ is contained in $X_{\lambda}^{(3)}$ as are $x_m=0$, $x_2=\pm I x_3$, $m \leq 2$ and $x_m=0$, $x_1=\pm I x_0$, $m\in \{2,3\}$. These are 12 curves, but generate a rank 9 sublattice of the Picard lattice, and this lattice contains the hyperplane class. Linear combinations of these curves span~$C$.
 \item[$i=6$] As in the case $i=2$ in this case we have that the automorphism $\sigma$ f\/ixes only six eigenvectors of the form $\tilde{\omega}_{\mathbf{k}}-\tilde{\omega}_{\sigma(\mathbf{k})}$. The seventh eigenvector is the class of the curve $x_3=0$, $x_0^2+x_1^2$, which is an element of~$C$. The other coordinate hyperplanes yields three further curves, contributing another two to the Picard number.
\item[$i=7$] In this case we take $x_0=0$ or $x_1=0$ then we f\/ind $x_3(x_2^3+x_3^3)=0$. In this way we f\/ind 8 lines, contributing six to cohomology.
\end{itemize}
\end{Example}

 We discuss now the other f\/ive examples:

\begin{Example} Consider now the case $i=4$. In this case the Fermat cover $Y_{\lambda}^{(4)}$ has degree~28. Let $G$ be the associated group of torus automorphisms. Then the multiples $\omega_{\mathbf{k}}$ with $\mathbf{k}$ a~multiple of $(7,7,7,7)$ generate a~rank~3 subspace of $H^3\big(V_{\lambda}^{(4)}\big)^G$ which is common to the examples $i=1,2,3$. The other monomial types associated with forms in $H^3(V_{\lambda}^{(4)})^G$ are
\begin{gather*} (7,11,15,23),\ (7,3,27,19),\ (14,2,18,22), \ (14,6,26,10), \ (21,1,9,25), \ (21,5,17,13)\end{gather*}
and those obtained by a cyclic permutation of the last three coordinates. In particular, these~$\omega_{\mathbf{k}}$ generate already a rank~21 subspace of $H^3(V_\lambda)^G$, which has dimension at most~21. Hence we found a basis for $H^3\big(V_{\lambda}^{(4)}\big)^G$ and we have $C=0$ in this case.

One can obtain some information on the zeta function as follows. If $q\equiv 1 \bmod 28$ then we can factor the zeta function over~$\mathbf{Q}_q$ according to strong equivalence classes (cf.\ \cite[Section~4]{Zetafamrec}). The strong equivalence class of $(7,7,7,7)$ consists further of $(14,14,14,14)$ and $(21,21,21,12)$. The class of $(7,11,23,15)$ consists further of $(14,18,2,22)$ and of $(21,25,9,1)$. The class of $(7,3,19,27)$ consists further of $(14,10,26,6)$ and of $(21,17,5,13)$. The other classes can be obtained by permutation the last three coordinates. In particular, we f\/ind that the characteristic polynomial on $H^3\big(V_{\lambda}^{(4)}\big)^G$ can be written as $P_3Q_3^3R_3^3$, where $P_3$, $Q_3$ and~$R_3$ are in $\mathbf{Q}_q[T]$ and have degree 3. The polynomial $P_3$ is the common factor and by Corollary~\ref{corMain} in~$\mathbf{Q}[T]$. Since $P_3Q_3^3R_3^3\in \mathbf{Q}[T]$ we f\/ind that also $Q_3(T)R_3(T)$ is in $\mathbf{Q}[T]$.

The results from \cite[Section 5]{zetafam} yield three explicit matrices, each $3\times 3$, whose entries are rational functions of generalised $p$-adic hypergeometric functions, such that the three corresponding characteristic polynomials are $P_3, Q_3$ and $R_3$.

As mentioned in the introduction of this paper, we did not f\/ind a complete set of generators for generic Picard group for two of the ten families. This family is one of these two families.
\end{Example}

\begin{Example} The degree of the Fermat cover of the f\/ifth example is 80. The monomial type $(20,20,20,20)$ and its two multiplies in $H^3\big(V_{\lambda}^{(5)}\big)$ yield the factor common with the examples $i=1,2,3,4$.

The other monomial types associated with classes in $H^3\big(V_{\lambda}^{(5)}\big)^G$, are \begin{gather*}(4,52,36,58),\ (24,72,56,8),\ (44,12,76,28),\ (64,32,16,48)\end{gather*} and the cyclic permutations of these.
Hence $W_\lambda^{(5)}$ has dimension 16. The subspace~$C$ has dimension~2 and contains the classes of the lines $x_0=x_2=0$ and $x_1=x_3=0$.

To f\/ind the curves contributing to the rank 16 part, we can use permutations, similarly as in the above examples. The cyclic permutation of $1$, $2$, $3$, $4$ is odd. Denote this permutation by~$\sigma_0$. The quotient by this permutation is a del Pezzo surface of degree~5.

Fix now a primitive f\/ifth root of unity $\zeta$. Let
\begin{gather*} \rho:=(x_0,x_1,x_2,x_3)\mapsto \big(\zeta x_0, \zeta^3 x_1,\zeta^4 x_2,\zeta^2 x_1\big).\end{gather*} For $i=1,2,3,4$ we set $\sigma_i=\sigma_0\rho^i$. Then each $\sigma_i$ has order 4.

Let $\mathbf{m}$ be a monomial types such that $\tilde{\omega}_{\mathbf{m}}$ is pulled back to one of the 16 forms $H^3\big(V_{\lambda}^{(5)}\big)^G$ not a multiple of $(20,20,20,20)$.

Consider now $\big\{ \sum\limits_{j=0}^3 \sigma_i^j \tilde{\omega}_{\mathbf{m}} \colon i=0,\dots,3\big\}$. A~direct calculation using a Vandermonde determinant shows that these four forms are linearly independent and that their span contains $\tilde{\omega}_{\mathbf{k}}$. Hence $\tilde{\omega}_{\mathbf{k}}$ is contained in the subspace spanned by $\mathop{\cup}\limits_{i=1}^4 H^3(U_\lambda)^{\sigma_i}$. So each of the $\tilde{\omega}_{\mathbf{k}}$ can be expressed as a linear combination of curves on the del Pezzo surface~$X_{\lambda}^{(5)}/\sigma_i$, with $i=0,1,2,3$.

Using the terminology of \cite[Section~6]{zetafam} we have two weak equivalence classes of monomial types, one consisting of three monomial types and consisting of 16 monomial types. The large class decomposes in four strong equivalence classes. These four strong equivalence classes are in one $\sigma$-orbit. From this we obtain if $q\equiv 1 \bmod 20$ then the characteristic polynomial of Frobenius is $P_4 Q_4^4$, where both~$P_4$ and~$Q_4$ are of degree~4.

A dif\/ferent approach to f\/ind curves on $X_{\lambda}^{(5)}$ would be to use the line $x_0=x_2=0$ to f\/ind an elliptic f\/ibration. A Weierstrass equation for this f\/ibration is
 \begin{gather*} y^2=x^3 -27s^4\big(\lambda^4+144\big)x-54s\big({-}s^5\lambda^6+864s^{10}+648s^5\lambda^2+864\big).\end{gather*}
For general $\lambda$ the f\/ibration has 2 f\/ibers of type~$II$ and 20 f\/ibers of type~$I_1$. The sections of this f\/ibration and the f\/iber class generate the Picard group for general~$\lambda$.
 \end{Example}

\begin{Example} For $i=8$ that the Fermat cover has degree 12, and the deformation vector is pulled back to $4242$. The Picard--Fuchs equation has order 4. The other monomial types $\omega_{\mathbf{k}}$ are built up from pairs from $42$, $45$, $48$, $4B$, $81$, $84$, $87$, $8A$ (where $A=10$, $B=11$) such that the entries add up to a multiple of $12$ and such that $\mathbf{k}$ is not a multiple of~$4242$. In total we f\/ind 12 such forms.

The complementary f\/ive-dimensional subspace comes from coordinate plane sections, i.e., $x_1=0$ yields $x_3\big(x_2^3+x_3^3\big)$, and also $x_3=0$ contributes. The total contribution is~5. We do not have any odd permutation to work with. However, this surface has many elliptic f\/ibrations and one may be able to work with them.

As mentioned in the introduction of this paper, we did not f\/ind a complete set of generators for generic Picard group for two of the ten families. This family is one of these two families.
\end{Example}

\begin{Example} In the ninth example we have that the common cover has degree~36. The deformation monomial has exponents $9$, $12$, $8$, $7$. There are 18 multiples of this vector without a zero in $\mathbf{Z}/36\mathbf{Z}$. Hence the Picard--Fuchs equation has degree 18. Moreover, the curves $x_2=0$, $x_0=i^kx_4$ together with the hyperplane class generate the generic Picard group.
\end{Example}

\begin{Example} In the tenth examples we have that the Fermat cover has degree 108. The deformation monomial has exponents $36$, $24$, $28$, $20$. There are 18 multiples of this vector without a zero in $\mathbf{Z}/108\mathbf{Z}$. Hence the Picard--Fuchs equation has degree~18. Moreover, we have the curves $x_1=x_3=0$, $x_1=0$, $x_2=\omega^i x_3$ and $x_3=0$, $x_0^3-x_1^2x_2=0$. This are f\/ive curves admitting two relations.
\end{Example}

\begin{Remark}In two cases we did not f\/ind generators. In these two cases dif\/ferent there is no permutation $\sigma$ of the coordinates which is automorphism of the family and such that the quotient surface is a rational surface. In the other examples with nontrivially $W_\lambda$, this space was generated by pull backs of curves coming from rational surfaces.

It is the author's experience that in characteristic zero, establishing explicit curves generating the Picard group of a surface, is an easier problem when working with surfaces with $h^{2,0}=0$ then when working with surfaces with $h^{2,0}>0$. This can be partly explained by the fact that degrees and intersection numbers of generators of the Picard group are determined by the topology of the surface in the case $h^{2,0}=0$, but not in the case $h^{2,0}>0$.

A similar problem is determining a basis of the Mordell--Weil group of an elliptic~K3 surfaces (which is equivalent to determining generators for the N\`eron--Severi group of that surface). This turned out to be much simplif\/ied if the~$K3$ surface is in various ways the pull back of a rational elliptic surface. (E.g., see~\cite{ToMe,KloKuw,Scholten}.)
\end{Remark}

\appendix

\section{Bitangents to special plane quartics}\label{appBit}

In Section~\ref{secQua} we considered ten pencils of quartic surfaces. In Proposition~\ref{prpdelPezzo} we showed that f\/ive of these pencils are (each in multiple ways) double covers of pencils of del Pezzo surfaces of degree two and we showed how the knowledge of the Picard group of these del Pezzo surfaces is suf\/f\/icient to determine the generic Picard group of each pencil. In this section we explain how one can f\/ind explicit generators for the Picard group of these del Pezzo surfaces. It is well-known that such a surface is a double cover of~$\mathbf{P}^2$ ramif\/ied along a quartic curve.

If the quartic curve is smooth then its has 28 bitangents. These bitangents are pulled back to two lines on the del Pezzo surface, and these lines generate the Picard group.

In order to f\/ind explicit equations for the del Pezzo surfaces of degree 2 and the quartic curves we are going to make the steps from the proof of Proposition~\ref{prpdelPezzo} explicit. This proposition applies only to $X^{(i)}_{\lambda}$ with $i\in \{1,2,3,6,7\}$, hence we concentrate on these cases. To ease the calculations we start by decomposing the def\/ining polynomials for $X_\lambda^{(i)}$ in sums of two polynomials. Therefore def\/ine the following polynomials
 \begin{gather*} f_1(x_0,x_1,x_2,x_3):=x_0^4+x_1^4+\lambda x_0x_1x_2x_3,\\
 f_2(x_0,x_1,x_2,x_3):=x_0x_1\big(x_0^2+x_1^2\big)+\lambda x_0x_1x_2x_3 ,\\
 g_1(x_2,x_3):=x_2^4+x_3^4, \\
 g_2(x_2,x_3):=x_2x_3\big(x_2^2+x_3^2\big), \\
 g_3(x_2,x_3):=x_2^3x_3+x_3^4,\\
h_1(u,v,x_2,x_3):=u^4-4u^2v+2v^2+\lambda vx_2x_3,\\
h_2(u,v,x_2,x_3):=v\big(u^2-2v\big)+\lambda vx_2x_3.\end{gather*}

The f\/ive pencils of quartic surfaces under consideration are def\/ined by the vanishing of
\begin{gather*} f_1+g_1,\quad f_1+g_2,\quad f_2+g_2, \quad f_1+g_3, \quad f_2+g_3.\end{gather*}

\subsection[$S^{(i,1)}_{\lambda}$]{$\boldsymbol{S^{(i,1)}_{\lambda}}$}
As we noted in the proof of Proposition~\ref{prpdelPezzo} each of these families is invariant under the automorphism $\sigma\colon (x_0,x_1,x_2,x_3)\mapsto (x_1,x_0,x_2,x_3)$.

In particular, each of the def\/ining polynomials is also a polynomial in $x_0+x_1$, $x_0x_1$, $x_2$, $x_3$. We def\/ined $h_1$, $h_2$ such that \begin{gather*}h_j(x_0+x_1,x_0x_1,x_2,x_3)=f_j(x_0,x_1,x_2,x_3).\end{gather*} Therefore the quotient $S^{(i,1)}_{\lambda}$ of $X^{(i)}_{\lambda}$ by $\sigma$ is the zeroset of
\begin{gather*} h_1+g_1,\quad h_1+g_2,\quad h_2+g_2, \quad h_1+g_3,\quad h_2+g_3\end{gather*}
in $\mathbf{P}(1,2,1,1)$. These polynomials def\/ine f\/ive families of surfaces in $\mathbf{P}(1,2,1,1)$. The general member is a del Pezzo surface of degree 2. The rational map $\mathbf{P}(1,2,1,1)\dashrightarrow \mathbf{P}^2$ def\/ined by $(u:v:x_2:x_3)\to (u:x_2:x_3)$ is def\/ined on all of $S^{(i,1)}_{\lambda}$. It establishes this surfaces as a double cover of $\mathbf{P}^2$ ramif\/ied along the zeroset of~$q_i$, the discriminant~$q_i$ of the def\/ining polynomial of~$S^{(i,1)}_{\lambda}$ considered as polynomial in~$v$. These discriminant are straightforward to compute. We list them here:
\begin{gather*}
q_1:=-8x_2^4+\lambda^2 x_2^2x_3^2-8\lambda x_2x_3u^2-8x_3^4+8u^4,\\
q_2:=-8x_2^3x_3+\lambda^2x_2^2x_3^2-8x_2x_3^3-8\lambda x_2x_3u^2+8u^4,\\
q_3:=8x_2^3x_3+\lambda^2x_2^2x_3^2t+8x_2x_3^3+2\lambda x_2x_3u^2+u^4,\\
q_6:=-8x_2^4+\lambda^2x_2^2x_3^2-8x_2^3x_3-8\lambda x_2x_3u^2+8u^4,\\
q_7:=8x_2^4+\lambda^2x_2^2x_3^2+8x_2^3x_3+2\lambda x_2x_3u^2+u^4.
\end{gather*}
Our aim is to f\/ind the bitangents to these curves and then pull them back to $X^{(i)}_{\lambda}$. If $\lambda$ is chosen such that the quartic curve is smooth then there are 28 bitangents. We start by looking for bitangents of the shape $u=a_2x_2+a_3 x_3$. Such a line is a bitangent to the curve~$q_i=0$ if we can f\/ind further $b$, $c$ such that the following polynomial vanishes
\begin{alignat*}{3}& q_1(a_2x_2+a_3x_3,x_3,x_2)-8\big(a_2^4-1\big) \big(x_2^2+bx_2x_3+cx_3^2\big)^2\qquad && \text{if} \quad i=1,&\\
&q_2(a_2x_2+a_3x_3,x_3,x_2)-8a_2^4 \big(x_2^2+bx_2x_3+cx_3^2\big)^2 \qquad &&\text{if} \quad i=2,&\\
&q_3(a_2x_2+a_3x_3,x_3,x_2)-a_2^4\big(x_2^2+bx_2x_3+cx_3^2\big)^2 \qquad &&\text{if} \quad i=3,&\\
&q_6(a_2x_2+a_3x_3,x_3,x_2)-8a_2^4\big(x_2^2+bx_2x_3+cx_3^2\big)^2 \qquad &&\text{if} \quad i=6,&\\
& q_7(a_2x_2+a_3x_3,x_3,x_2)-a_2^4 \big(x_2^2+bx_2x_3+cx_3^2\big)^2 \qquad &&\text{if}\quad i=7. & \end{alignat*}
The factors $8\big(a_2^ 4-1\big)$, $8a_2^ 4$ and $a_2^4$ are chosen in order to kill the coef\/f\/icient of $x_3^4$ in each of the polynomials. Hence each of the f\/ive above polynomials is a polynomial of degree 3 in $x_3$. These polynomials can be computed with the help of some computeralgebra package. Unfortunately, the obtained expressions are too long to include them here. From these calculations one deduces that both the coef\/f\/icient of $x_3^3$ and of~$x_3^2$ are linear in~$b$ and~$c$. We can solve for $b$ and $c$ and substitute the result. However, in order to solve for~$b$ and~$c$ we have to divide by $a_2^4-1$ if $i=1$ and by $a_2$ in the other cases, hence for the moment we have to assume that they are nonzero.

We are then left with two nonzero coef\/f\/icients. The coef\/f\/icient of $x_3^0$ is a cubic in $a_3$. Eliminating $a_3$ leaves a polynomial of degree either $20$ (if $i=1$) or $24$ (if $i\neq 1$) in~$a_2$, which we list below.

Each of the zeroes yields a possible value for~$a_2$. One easily checks that each value for $a_2$ determines a unique value for~$a_3$. In this way we f\/ind $20$ or $24$ bitangents. Note that each of the f\/ive families admits the automorphism $(u,v,x_2,x_3)\mapsto (u,v,-x_2,-x_3)$. This implies that if~$(a_2,a_3)$ def\/ines a bitangent then so does~$(-a_2,-a_3)$. Hence the f\/inal polynomial in~$a_2$ is actually a polynomial in~$a_2^2$. Depending on the case there are further automorphisms, which could give further simplif\/ications.

We now list for each case the degree 24 polynomial in $a_2$. The case $i=1$ is slightly more involved then the other ones, so we start with the case $i\geq 2$.

For $i=2$ we f\/ind that if $a_2$ is a zero of
 \begin{gather*}
 \big(\big(\lambda^2+16\big)a_2^4+4\lambda a_2^2\lambda+2\big)\big(\big(\lambda^2-16\big)a_2^4+4\lambda a_2^2\lambda+2\big) \cdots \\
\qquad {} \cdots \big(1024 a_2^{16}+128 \lambda^3 a_2^{14}+\big(2\lambda^6+960\lambda^2\big)a_2^{12}+\big(20\lambda^5+2560\big) a_2^{10}+\cdots\\
\qquad{} \cdots + \big(73\lambda^4 +2176\big) a_2^{8}+120\lambda^3a_2^{6}+92\lambda^2 a_2^{4}+32\lambda a_2^{2}+4\big) \end{gather*}
then there is a unique $a_3$ yielding a bitangent. The degree 16 factor can be written as the product of two factors of degree 8 over $\mathbf{F}_q\big(\sqrt{2}\big)$. This yields 24 of the 28 bitangents.

 For $i=3$ we f\/ind that $a_2$ is a zero of
\begin{gather*}
\big(2a_2^4+ta_2^2+2\big)\big(2a_2^4-ta_2^2-2\big) \cdots \big(a_2^8+4a_2^6+\big( \lambda^2-4\lambda+8\big)a_2^4+(4\lambda-8)a_2^2+4\big) \cdots \\ \qquad{} \cdots \big(a_2^8-4a_2^6+\big( \lambda^2+4\lambda+8\big)a_2^4+(4\lambda+8)a_2^2+4\big) \end{gather*}
then there is a unique $a_3$ yielding a~bitangent. Each of the two degree 8 factors is a product of two factors of degree 4 factors over $\mathbf{F}_q\big(\sqrt{-1}\big)$. This yields~24 of the 28 bitangents.

For $i=4$ we f\/ind that $a_2$ is a zero of
 \begin{gather*}
\big({-}512\lambda^{6}-1769472\big)a_2^{24}-9216\lambda^{5}a_2^{22} +\big(2\lambda^{10}-62976\lambda^{4}\big)a_2^{20}+ \cdots \\
\qquad{} \cdots \big(36\lambda^{9}-286720\lambda^{3}\big)a_2^{18}+\big(273\lambda^{8}-663552\lambda^{2}\big)a_2^{16}+ \cdots \\
\qquad{} \cdots +\big(1136\lambda^{7}-700416\lambda\big)a_2^{14} +\big(2840\lambda^{6}-276480\big)a_2^{12}+4416\lambda^{5} a_2^{10}+ \cdots \\
\qquad {}\cdots +4312\lambda^{4} a_2^8+2624\lambda^{3}a_2^6+960\lambda^{2}a_2^4+192\lambda a_2^2+16\end{gather*}
 then there is a unique $a_3$ yielding a bitangent. This yields 24 of the 28 bitangents.

For $i=5$ we f\/ind that $a_2$ is a zero of
 \begin{gather*} \big({-}3a_2^8+2\lambda a_2^6 +\lambda^2 a_2^4+12a_2^4+4\lambda a_2^2+4\big) \cdots \\
 \qquad{} \cdots \big(9a_2^{16}+6\lambda a_2^{14} +\big(7\lambda^2 a_2^{12}+36\big)a_2^{12}+\big(60-2\lambda^3\big)a_2^{10}\lambda + \cdots \\
 \qquad{} \cdots +\big(\lambda^4-20\lambda^2+156\big)a_2^8+\big(8\lambda^3-56\lambda\big)a_2^6 +\big(24\lambda^2-48\big)a_2^4+32a_2^2\lambda +16\big) \end{gather*}
then there is a unique $a_3$ yielding a bitangent. This yields~24 of the 28 bitangents. If $\omega^2=-\omega-1$ then over $\mathbf{F}_q(\omega)$ we can write the degree 16 factor as a product of two factors of degree~8.

To f\/inish the cases $i=2,3,6,7$ we need to f\/ind~4 further bitangents. The above approach gives all bitangents of the form $u=a_2x_2+a_3x_3$ with $a_2\neq 0$. It turns out that there are no bitangents with $a_2=0$, however there are bitangents of the form $a_2x_2+a_3x_2=0$. One easily sees that the line $x_3=0$ is a hyperf\/lex line (and therefore a bitangent) and that the remaining three bitangents are of the form $x_2=ax_3$, with $a$ a zero of
\begin{alignat*}{3} & a\big(8a^2+\lambda^2 a +8\big) &&\qquad \text{if} \quad i=2,&\\
& a\big(a^2+1\big) \qquad &&\text{if} \quad i=3,&\\
 &8a^3+ \lambda^2 a^2+2 \qquad &&\text{if} \quad i=6,&\\
& (a+1)\big(a^2-a+1\big) \qquad &&\text{if}\quad i=7. &\end{alignat*}

In the cases $i=2,3,6,7$ we f\/ind that 24 of the bitangents can be described in terms of a~polynomial in~$a_2$ of degree~$24$. Since $a_1$ is the unique root of a polynomial with coef\/f\/icients in $\mathbf{F}_q(\lambda,a_2)$ we f\/ind that $a_1\in \mathbf{F}_q(\lambda,a_2)$. The equations def\/ining $c$ and $d$ are linear, hence they are also in $\mathbf{F}_q(\lambda, a_2)$. Hence the bitangent is def\/ined over~$\mathbf{F}_q(\lambda,a_2)$. However, the lines on the del Pezzo surface may be def\/ined over a degree~2 extension. If $\ell=V(-u+ax_2+bx_3)$ then the equation for the del Pezzo surface is a quadratic equation in $v$. It restriction to $u=a_1x_2+a_2x_3$ is a quadratic equation with discriminant $q_i|_{\ell}$. This discriminant is of the form $C_i\big(x_2^2+b_ix_2x_3+c_ix_3^2\big)^2$. Hence to def\/ine each of the two corresponding lines on~$S^{(1,i)}_\lambda$ we need to take a square root of $C$. An explicit calculation now show that $C$ depends only on~$i$ and~$a_2$. More precisely, we have that $C_i$ equals $8a_2^4$, $a_2^4$, $8a_2^4$, $a_2^4$ for $i=2,3,6,7$. Hence for $i=3,7$ both lines are def\/ined over $\mathbf{F}_q(\lambda,a_2)$ but for $i=2,6$ they are def\/ined over $\mathbf{F}_q\big(\lambda,a_2,\sqrt{2}\big)$.

Similarly, one easily checks that the f\/lex line is def\/ined over $\mathbf{F}_q$ and that the two corresponding lines on the del Pezzo surface are def\/ined over $\mathbf{F}_q$ if $i=3,7$ and over $\mathbf{F}_q\big(\sqrt{-1}\big)$ if $i=2,6$ and that each of the remaining lines are def\/ined over $\mathbf{F}_q(a)$ if $i=3,7$ and $\mathbf{F}_q\big(\sqrt{2},\lambda,a\big)$ if $i=2,6$.

For $i=1$ we can copy the above approach, but in the f\/irst step we f\/ind 20 rather than 24 bitangents of the form $u=a_2x_2+a_3x_3$. These 20 bitangents are one of the following (where $I$ is a f\/ixed root of~$-1$):
\begin{enumerate}\itemsep=0pt
\item[1)] $a_3=a_2$ and $a_2^4 \big(\lambda^2+16\big)-16\lambda a_2^2+\lambda^2 +16=0$,
\item[2)] $a_3=-a_2$ and $a_2^4 \big(\lambda^2+16\big)+16\lambda a_2^2+\lambda^2 +16=0$,
\item[3)] $ a_3=Ia_2$ and $a_2^4\big(\lambda^2-16\big)-16 I \lambda a_2^2+\lambda^2-16=0$,
\item[4)] $ a_3=-Ia_2$ and $a_2^4\big(\lambda^2-16\big)-16 I \lambda a_2^2+\lambda^2-16=0$,
\item[5)] $ a_3=\frac{t}{4a_2}$ and $a_3^4=1$.
\end{enumerate}
Using symmetry we f\/ind that further bitangents are given by $ a_3=\frac{t}{4a_2}$ and $a_2^4=1$.

There are four further bitangents of the form $x_2=ax_3$ with
\begin{gather*} 8a^4+\lambda^2 a^2+8=0.\end{gather*}
One easily checks that the corresponding lines on the del Pezzo surface are def\/ined over the f\/ield $\mathbf{F}_q\big(\lambda, a_2,\sqrt{2(a_2^4-1)}\big)$ (if $a_2^4\neq 1$), over $\mathbf{F}_q\big(\lambda, a_1,\sqrt{2(a_3^4-1})\big)$ (if $a_3^4\neq 1$) and over $\mathbf{F}_q\big(\lambda, a,\sqrt{2}\big)$ (for the f\/inal four lines).

\subsection[$S^{(i,2)}_{\lambda}$ and $S^{(i,3)}_{\lambda}$]{$\boldsymbol{S^{(i,2)}_{\lambda}}$ and $\boldsymbol{S^{(i,3)}_{\lambda}}$}
Once we found the lines on $S^{(i,1)}_\lambda$ we can use them to f\/ind also the lines on $S^{(i,2)}_\lambda$ and $S^{(i,3)}_\lambda$.

Let
\begin{gather*} h_3(u,v,x_2,x_3):=u^4+4u^2v+2v^2+\lambda vx_2x_3, \qquad h_4(u,v,x_2,x_3):=v\big(u^2+2v\big)+\lambda vx_2x_3.\end{gather*}
Proceeding as above we f\/ind that that $S^{(i,2)}_{\lambda}$ is def\/ined by
\begin{gather*} h_3+g_1,\quad h_4+g_2,\quad h_4+g_2,\quad h_3+g_3,\quad h_4+g_3.\end{gather*}
The map $(u,v,x_2,x_3)\mapsto (uI,v,x_2,x_3)$ def\/ines an isomorphism $S^{(i,2)}_{\lambda} \to S^{(i,1)}_{\lambda}$ for $i=1,2,6$. The map $(u,v,x_2,x_3)\mapsto (uI,-v,Ix_2,Ix_3)$ def\/ines an isomorphism $S^{(i,2)}_{\lambda} \to S^{(i,1)}_{\lambda}$ for $i=3,7$.

Let $h_5(u,v,x_2,x_3):=u^4-4 I u^2v-2v^2+\lambda vx_2x_3$ then $S^{(i,3)}_{\lambda}$ is def\/ined (for $i=1,2,6$) by
\begin{gather*} h_5+g_1,\quad h_5+g_2,\quad h_5+g_3.\end{gather*}
The map $(u,v,x_2,x_3)\mapsto (u,Iv,x_2,Ix_3)$ def\/ines an isomorphism $S^{(i,3)}_{\lambda} \to S^{(i,2)}_{\lambda}$ for $i=1$. The map $(u,v,x_2,x_3)\mapsto \big(u,Iv,\zeta x_2,\zeta^5 x_3\big)$ def\/ines an isomorphism $S^{(i,3)}_{\lambda} \to S^{(i,2)}_{\lambda}$ for $i=2$.

For $i=7$ we need also to act on $\lambda$: The map $(u,v,x_2,x_3)\mapsto (u,Iv, x_2, x_3)$ def\/ines an isomorphism $S^{(i,3)}_{-I\lambda} \to S^{(i,2)}_{\lambda}$ for $i=2$.

Substituting $v=x_0x_1$ and $u=x_0+x_1$ (if $j=1$), $u=x-y$ (if $j=2$) or $u=x+Iy$ (if $j=3)$ in the equations of a line on $S^{(i,j)}_{\lambda}$ then yields the corresponding conic on $X_\lambda^{(i,j)}$.

\subsection*{Acknowledgements}

The author would like to thank John Voight and Tyler Kelly for various conversations on this topic. The author would like to thank the referees for various suggestions to improve the exposition.

\pdfbookmark[1]{References}{ref}
\LastPageEnding

\end{document}